\documentclass[12pt]{amsart}
\usepackage[colorlinks=true, pdfstartview=FitV, linkcolor=blue, citecolor=green, urlcolor=black,filecolor=magenta]{hyperref}
\usepackage{graphicx,verbatim,amsmath,amssymb}

\newcommand{\aaa}{{\mathcal A}}
\newcommand{\bbb}{{\mathcal B}}

\newcommand{\lcal}{{\mathcal L}}
\newcommand{\ppp}{{\mathcal P}}
\newcommand{\sss}{{\mathcal S}}
\newcommand{\rrr}{{\mathcal R}}

\newcommand{\numa}{{|\mathcal A|}}
\newcommand{\numb}{{|\mathcal B|}}

\newcommand{\T}{{\bf T}} %proxy for tiling
\newcommand{\Tsp}{{\Omega}} %proxy for tiling space
\newcommand{\Sol}{{\mathbb S}} % proxy for solenoid
\newcommand{\tile}{{t}} %proxy for tile
\newcommand{\patch}{{P}} %proxy for patch
\newcommand{\protoset}{{\mathcal{P}}} %proxy for prototile set
\newcommand{\M}{{A}} %proxy for substitution matrix
\newcommand{\R}{{\mathbb R}}

\newcommand{\N}{{\mathbb N}}

\newcommand{\Z}{{\mathbb Z}}
\newcommand{\C}{{\mathbb C}}
\newcommand{\vece}{{\vec{e}}}

\newcommand{\vecv}{{\vec{v}}}

\newcommand{\transversal}{{\Xi(\Tsp)}}

\def\supp{\text{supp}}
\def\lab{\text{label}}

\newtheorem{thm}{Theorem}[section]
\newtheorem*{thm*}{Theorem}
\newtheorem{cor}[thm]{Corollary}
\newtheorem{lem}[thm]{Lemma}
\newtheorem{prop}[thm]{Proposition}
\newtheorem{dfn}[thm]{Definition}

\everymath{\displaystyle}

\theoremstyle{definition}
\newtheorem{ex}{Example}

\begin{document}

\title{Tilings with infinite local complexity}
\author{Natalie Priebe Frank}
\date{\today}
\address{Natalie Priebe Frank\\Department of Mathematics\\Vassar
  College\\Poughkeepsie, NY 12604} \email{nafrank@vassar.edu}
  
  \thanks{Some work presented here was done in collaboration with Lorenzo Sadun and Ian Putnam, both of whom the author thanks for their hospitality and many illuminating discussions.}
\maketitle

\begin{abstract}
This is a chapter surveying the current state of our understanding of tilings with infinite local complexity.  It is intended to appear in the volume {\em Directions in Aperiodic Order}, D. Lenz, J. Kellendonk, and J. Savienen, eds, Birkhauser.
\end{abstract}

\section{Introduction}
%\subsection{History} 

Most of the literature on tiling spaces and their dynamical systems has focused on those with finite local complexity (FLC).  In this paradigm there is a finite set $\ppp$ of tiles called `prototiles', congruent copies of which are used to cover the plane (or $\R^d$) without gaps or overlaps.  Moreover, the adjacencies between tiles are restricted so that there are only finitely many two-tile configurations.  If there can be infinitely many two-tile configurations in a tiling, then that tiling is said to have infinite local complexity (ILC).  

When tilings are looked at from a physical perspective it makes sense to consider not just individual tilings but rather spaces whose elements are tilings that share some common properties.  These tiling spaces are given a metric topology where the distance between two tilings is defined by how similar they are in balls around the origin (see Section \ref{metric.section} for precision).  When there are only finitely many two-tile configurations that are found in any tiling of a tiling space, we say that the tiling space itself has finite local complexity; otherwise, it has infinite local complexity.  FLC tiling spaces have been the standard objects used to model the atomic structure of crystals and quasicrystals and have proved quite effective in the study of statistical properties, diffraction patterns, and energy spectra of aperiodic solids.  

A tiling space that is of finite local complexity can be homeomorphic to one with infinite local complexity \cite{Radin-Sadun}.  Thus finite local complexity is not a topological invariant and should not be considered an intrinsic property when topological methods are used to study aperiodic tilings.

 Examples of tilings with infinite local complexity have appeared sporadically \cite{Danzer,Me.Robbie,Me.Lorenzo,Kenyon.rigid,Radin.pinwheel,Sadun.pin}, and it is increasingly clear that the class isn't as unnatural as previously imagined.  Moreover, most of the `usual' FLC tools and techniques can be used in the ILC case, and one of the goals of this chapter is to explain exactly how to adapt the existing machinery.  We take as fundamental the requirement that prototiles come from {\em compact}, not necessarily finite, sets.  This means that both the `supports' of the tiles (i.e., their underlying sets in $\R^d$) and the `labels' of the tiles (which are used to distinguish tiles with congruent supports) must come from compact sets.  We will see that this fundamental requirement means that ILC tiling spaces are compact (see Section \ref{TilingSpaces.sec}).  We delay formal definitions until Section \ref{sec:tiling.defs} and provide some informal examples now.
  
 \subsection{Introductory examples}
 Since one-dimensional tiles are closed intervals, any tiling made from a finite number of interval lengths with a finite number of labels must have finite local complexity.   So  in order to have infinite local complexity in one dimension there must be either an infinite label set or an infinite number of lengths (or both).  
 
 \begin{ex}A first example is to allow tiles to take lengths from some closed interval, for instance we could require that $1 \le $ length $\le 3$.
We can let the support of a prototile $p_x$ be the interval $[0,x]$, and we can label the tile by its length, $x$.  It is convenient to omit the label when it is possible to tell tiles apart by their supports, as is the case here, but we are including them for consistency with the definitions provided in Section \ref{sec:tiling.defs}.  The prototile $p_x$ is formally the pair $([0,x],x)$, and the prototile set in this example 
is thus $\ppp=\{p_x, x \in [1,3]\}$.  Notice that the set of supports of prototiles is compact in the Hausdorff metric and the set of labels is compact in the usual distance metric in $\R$.  Importantly, if a sequence of prototiles has a convergent label sequence, then their supports converge as well.  This makes it possible to say that the sequence of prototiles themselves converge.

A tile is simply a translate of $p_x$ by some element $y \in \R$; we write $t=p_x +y = ([y,y+x],x)$.  (Note that translation changes the support of a tile but not its label.) We could make a tiling from  such tiles in any number of ways, for instance by generating a sequence of random numbers in $[1,3]$ and laying down tiles of those lengths in any order.  With probability one such a tiling will have infinite local complexity because it has infinitely many different tiles and thus has infinitely many different two-tile patterns.  
 \end{ex}

This example provides a nice test case for computation since it is really different than the standard FLC situation, but still quite simple.  Example \ref{Variable.size.example} consists of a hierarchical tiling space based on a prototile set derived from $\ppp$.  We introduce its construction in Section \ref{substitution} and give it a thorough analysis in Section \ref{vss.analysis}.
  
\begin{ex} For another one-dimensional example we take a single interval length for the supports but allow for infinitely many labels.  Suppose that the support of every prototile is $[0,1]$, but that each prototile takes a label from some compact label set $\lcal$.  For concreteness, let $\lcal = S^1= \R/\Z$, the unit circle.  We cannot tell two prototiles apart by their supports, so the label tells us when two of them are different, and the distance between their labels tells us how different they are.  

An interesting way to construct a tiling from this prototile set is to fix an element $\alpha \in S^1$ consider the sequence of labels $x + n\alpha \mod 1$ for any $x \in \R$.  If $\alpha$ is irrational then the tilings generated by this label sequence will have infinite local complexity since $\{n \alpha \mod 1\}$ is infinite (and in fact uniformly distributed).
\end{ex}

\begin{ex}
A two-dimensional example of an ILC tiling can be constructed from unit squares.  Tile the plane in rows of tiles, but offset each row from the next by a randomly chosen number in [0,1].   With probability one, the result will be an ILC tiling.  A non-random variation on this theme is to base the offsets on some fixed irrational number $\alpha$.  Lay the first row of squares along the $x$-axis with a vertex at the origin.  Place an endpoint of the row at height $y=1$ at $x=\alpha$, and the endpoint of the row at height $y = n$ at $x = n \alpha$.  Since $\{n \alpha \mod 1\}$ is uniformly distributed in $[0,1]$, the offsets between rows will be too and in this way form an ILC tiling of the plane.
\end{ex}

\begin{ex}
A well-known example that has infinite local complexity up to translations is the pinwheel tiling.  Pinwheel tiles appear in infinitely many orientations in any individual pinwheel tiling and so there are not finitely many different two-tile patches
that are translates of one another.  This is a borderline case, however:  the tiles fit together in finitely many ways even though these allowed configurations appear in infinitely many orientations.
It is sometimes useful, then, to consider the pinwheel tiling space to be of finite local complexity by allowing rotations along with translations. 
\end{ex}

\subsection{Ways infinite local complexity arises.}

In higher dimensions there are many natural examples of tilings with infinite local complexity.
For instance, the atomic structure of an ideal crystal is modeled by a lattice of points, but the atoms in an actual crystal appear within a certain tolerance of that lattice.  A standard perspective to let the atomic structure generate a tiling, either by using the atoms as tile vertices or, by taking the Vorono\"{\i} tessellation of the set of atomic locations, or by some other method.   In the case of an ideal crystal, all methods yield periodic tilings with patches of tiles forming unit cells.   In the case of an actual crystal, however, these tiles will be deformed within a certain tolerance and we will have an infinite number of tile shapes.

Infinite local complexity has long been known to arise even when there are a finite number of tile shapes.  
There are tilings with a finite number of tile types inside of which `fault lines' develop.  Defined formally in Section \ref{Fault.and.fractured}, a fault line separates a tiling into half-tilings that can slide parallel to the fault line to produce new tilings from the tiling space.
The presence of fault lines often result in an infinite number of local adjacencies. 
If we were to encode adjacency information as labels for the tiles, then we would have infinitely many labels.   If a tiling of $\R^2$ has only a finite number of tile shapes up to Euclidean motions,  it is proved in \cite{Kenyon.rigid} that there are only two ways ILC can appear:  either along a fault line or along a fault circle.    The former case requires tiles that have a straight edge somewhere, while the latter requires tiles with an edge that is a circular arc of some given radius.

%\subsection{An array of questions}
%
%In any of these situations we may or may not have a tiling space that is homeomorphic to one with finite local complexity.  This question can be resolved for specific examples by appealing to Theorem \ref{One-D-FLCresult}.

\subsection{Outline of this chapter}
Section \ref{sec:tiling.defs} contains the details on how we conceive of tiles, tilings, tiling spaces, and the tiling metric in the presence of infinite local complexity.  Our definitions coincide with those for finite local complexity tilings when that condition is satisfied.

Section \ref{sec:ergodic} addresses basic analysis of ILC tiling dynamical systems. The translation dynamical system is defined and we explain what minimality, repetitivity, and expansivity mean in this context.  The notion of `cylinder sets' is adapted from symbolic and FLC dynamics, and we show how to deal with some subtle yet important details that impact how they are used.  We show how to think about translation-invariant measures and their relationship to patch frequency.  Finally we discuss how to generalize the notions of entropy and complexity to this situation.

Tilings with a hierarchical structure generated by substitution or fusion are the topic of Section \ref{sec:hierarchical}.  The construction methods adapt pretty much directly from the FLC case, except care must be taken to preserve compactness of supertile sets.  Transition matrices, so useful in frequency computations for FLC self-similar and fusion tilings, need to be dealt with as transition maps instead.  The idea of recognizability takes little work to adapt to the ILC case, but primitivity requires some care.

Existing results on ILC tiling spaces are collected into Section \ref{results.section}.  We give a `fault lines' a proper definition, and since they aren't topologically invariant we introduce the related idea of `fractured' tiling spaces.  The effect of fault lines and fractures on the topological spectrum is explained before we move on to results specific to the hierarchical tilings case.  The fact that primitivity continues to imply minimality is proved and conditions are given that make the converse true as well.  We also explain how to think about the invariant measures for fusion systems.  In the special case of fusion tilings with strictly finite supertile sets we show the similarity to FLC fusion tilings.  Finally we tell everything that is currently known on the important question ``When is an ILC tiling space homeomorphic to an FLC tiling space?"

In Section \ref{example.analysis} we apply our toolbox to three different examples.  A point of interest that does not appear elsewhere in the literature is how to see certain tilings (`direct product variations') as projections of stepped, branched surfaces in higher dimensions and how that can give rise to infinite local complexity.

The paper concludes with two main categories of questions about tilings with infinite local complexity.  One of these has already been mentioned, the question of when an ILC tiling space is homeomorphic, or even topologically conjugate to, an FLC tiling space.  The other type is about how the geometric and combinatorial aspects of the tiles or tilings affect the dynamical, measure-theoretic, or topological properties of their tiling spaces.

\section{Compact tiling spaces} \label{sec:tiling.defs}
\subsection{Tiles, patches, and tilings} 

There are two main ingredients for tiles in a tiling with infinite local complexity: supports and labels.  The support is the underlying set in $\R^d$ and the label can be thought of as distinguishing between tiles that have congruent supports, perhaps by color or by orientation.  Often it is convenient to more or less ignore the labels but since they are quite handy we include them as a fundamental part of our definition.  The support and label sets must work together in a precise way in order to define a coherent prototile set that can be used to construct infinite tilings via translation.

Let $\sss$ denote a set of subsets of $\R^d$, each of which is a topological disk containing the origin in its interior.  Assume $\sss$ is a compact metric space under the Hausdorff metric, in which case $\sss$ can serve as a set of prototile supports.  Let $\lcal$ be another compact metric space, to be used as the prototile label set.  Let $sp:\lcal \to \sss$ be a continuous surjection called the {\em support map} that assigns to each label a set in $\R^d$ that serves as the physical tile itself.

\begin{dfn}\label{prototile.def}
A {\em prototile} is a pair $p=(S,l)$, where $S \in \sss$, $l \in \lcal$, and $S = sp(l)$.  We call $S$ the {\em support} and $l$ the {\em label} of $p$.  A prototile set $\ppp$ is the set of all prototiles associated to a given label set, support set, and support map.
\end{dfn}

Since the support map $sp$ is continuous we have the property that if a sequence of labels converges in $\lcal$, their corresponding supports converge in $\sss$.  This will give us a way to talk about convergence of prototiles and compactness of the prototile set.

The primary action on tiles will be by {\em translation} by $x \in \R^d$.  If $p = (S,l) \in \ppp$ we define the {\em $\ppp$-tile} or just {\em tile} $t = p-x$ to be the pair $(S-x, l)$.  That is, we translate the support of $p$ to a different location but keep the label the same.   As for prototiles, tiles have supports and labels;  the support of the above tile $t$ is the set $\supp(t)=S-x$ and the label of $t$ is $l$.  Given an arbitrary tile $t$, we have support and label maps such that $\supp(t) \subset \R^d$ and $\lab(t) \in \lcal$. 

A handy concept in tiling theory is that of the {\em control point} of a tile $t = p - x$, where $p \in \ppp$ and $x \in \R^d$, which is defined simply to be the point $x$.  This point represents the location in $t$ of the origin in $p$ and gives us a point of reference for each tile.

\begin{dfn} A finite union of $\ppp$-tiles whose supports cover a connected region and intersect only on their boundaries is called a {\em patch}.   
\end{dfn}
We can write $P = \bigcup_{k = 1}^n t_k$, where $\bigcup_{k = 1}^n \supp(t_k)$ is connected and $\supp(t_i) \cap \supp(t_j)$ is either empty or contains only boundary points whenever $i \neq j$.
Like tiles, patches can be translated and we define $P-x =\bigcup_{k = 1}^n (t_k-x)$.   Two patches are said to be {\em equivalent} if they are translates of one another.

\begin{dfn} An infinite union of $\ppp$-tiles whose supports cover the entirety of $\R^d$ and whose pairwise intersections contain only boundary points is called a {\em tiling} $\T$. 
\end{dfn}
 Like patches and tiles, a tiling can be translated by an element $x \in \R^d$ by translating each tile of $\T$ by $x$.  This produces a new tiling we denote by $\T-x$.  Precisely, if $\T = \bigcup_{i \in \Z} \tile_i$ is a tiling expressed as a union of tiles, then we write $\T-x = \bigcup_{i \in \Z} (\tile_i - x)$, where $\tile_i - x = (\supp(\tile_i) - x,\lab(\tile_i))$. This results in an exact copy of the tiling $\T$, except moved so that what was at the point $x$ is now at the origin.

\subsection{Tile, patch, and tiling metrics} \label{metric.section} In order to understand tiling spaces we need to know how to measure the distance between tiles, patches, and tilings.    To simplify notation (but not add confusion, we hope) we will use $d(x,y)$ to denote distance where $x$ and $y$ are tiles, patches, or tilings.  Each builds on the last.

The distance between two tiles $\tile_1$ and $\tile_2$ is the maximum of the Hausdorff distance between the supports of the tiles and the difference between the labels:
 \begin{equation}
d(\tile_1,\tile_2) = \max(d_H(\supp(\tile_1), \supp(\tile_2)),d_L(\lab(\tile_1),\lab(\tile_2)))
\end{equation}
The distance between two patches $\patch_1$ and $\patch_2$ can be computed provided the tiles are in one-to-one correspondence.  
Suppose $G$ is the set of all bijections $f$ assigning a tile from $\patch_1$ to a tile from $\patch_2$. In this case we define
\begin{equation}
d(\patch_1,\patch_2) =min_{f \in G}\{ max_{\tile \in \patch_1}\{d(\tile,f(\tile))\} \}
\end{equation}
Intuitively, we take the bijection that makes the best fit between the two patches and then consider the maximum distance between tiles paired by the bijection.  In the FLC case patches are always matched up by a congruence, usually a translation, in which case the distance is the length of the translation vector.  In the ILC case it is necessary to let the tiles move independently from one patch to the other.

The metric for tilings is based on the patch metric and says that two tilings are close if they very nearly agree on a big ball around the origin.
For two tilings $\T_1$ and $\T_2$ we define
\begin{equation}
d(\T_1,\T_2) = inf_{\epsilon>0}\left\{\exists \patch_1\subset \T_1 \text{ and } \patch_2 \subset \T_2 \, | \, B_{1/\epsilon}(0) \subset \supp(\patch_i) \text{ and } d(\patch_1,\patch_2) < \epsilon \right\}
\end{equation}
%we define $d(\T_1, \T_2)$ to be the infimum of the set of $\epsilon > 0$ such that there is a $\T_1$-patch $P_1$ and a $\T_2$-patch $P_2$ such that both $P_1$ and $P_2$ contain $B_{1/\epsilon}(0)$ and $d(P_1, P_2) < \epsilon$.    
provided such an $\epsilon$ exists and is not greater than 1.
If there is no such $\epsilon$, or if  the infimum is greater than 1, we define the distance between the tilings to be 1.  

\subsection{Tiling spaces}\label{TilingSpaces.sec}  Rather than trying to study an individual tiling it often makes sense to study all tilings that have certain properties in common.  The standard way to do this, motivated by physical applications, is to construct a topological space of tilings.
\begin{dfn}
A {\em tiling space} $\Tsp$ is a set of tilings of $\R^d$ that is invariant under the action of translation and closed under the topology given by the tiling metric $d$.
\end{dfn}
\noindent One common way to make a tiling space is by taking the closure of the translational orbit of some fixed tiling $\T$, in which case we write $\Tsp_\T$.   This tiling space is called the {\em hull} of $\T$.    

\begin{thm} Tiling spaces are compact in the metric topology.\end{thm}
\begin{proof}
We establish sequential compactness for patch sets and then extend to tilings. The key to seeing this is to show that the set of all patches contained in a bounded region and having a fixed number $n$ of tiles is compact for every $n$.  Such a set of patches is parameterized by a bounded subset of $\ppp^n \times \R^{dn}$, where the elements of $\R^{dn}$ are the locations of the control points and thus lie in a bounded region. 
The individual tiles in any sequence of patches will have convergent subsequences since $\ppp$ is compact and the tiles lie in a bounded region.  We can diagonalize to get a sequence of patches for which all of the individual tiles converge; since each patch in the sequence is connected and the tiles have nonoverlapping boundaries, the limit will have this property as well.  Thus every sequence of $n$-tile patches in a bounded region has a convergent subsequence.  Sequential compactness for $\Tsp$ now follows by finding subsequences of tilings that have convergent sequences of patches covering larger and larger regions around the origin.
\end{proof}

\subsection{The transversal $\transversal$ of a tiling space.}  In definition \ref{prototile.def} we defined the prototile set as being a representative set of tiles located so that the origin lies in their support at a control point.
\begin{dfn}
The {\em transversal} $\transversal$ of a tiling space $\Tsp$ is the set of all tilings in $\Tsp$ with a control point at the origin.  
Put another way, $\transversal$ is the set of all tilings in $\Tsp$ containing a prototile.
 \end{dfn}
Every tiling in $\Tsp$ is the translation of lots of tilings from the transversal.%, and so we can consider the tiling dynamical system $(\Tsp, \R^d)$ to be the suspension of a groupoid action on $\transversal$.   
Moreover, every point in the tiling space has a neighborhood that is homeomorphic to an open set in $\R^d$ crossed with an open subset of the transversal.

Much of the work done on FLC tiling spaces uses the transversal in an essential way.  For instance, the $C^*$-algebra of a tiling space is strongly Morita equivalent to the $C^*$-algebra of its transversal.   This means the $K$-theory of the tiling space can be computed from the transversal.  By the gap-labelling theorem, we then understand the possible energy levels that the tiling space can support when considered as an atomic model.   The transversal also makes possible the definition of a Laplace-Beltrami operator that holds information on key mechanical properties of solids.  This has been studied in the FLC case in, for example, \cite{Julien-Savinien}; there is hope that this analysis can be extended to at least some tilings with infinite local complexity.

Thus it is important to understand the structure of the transversal.    When a tiling space has finite local complexity, the transversal is always totally disconnected and, under the condition of repetitivity, is a Cantor set.   Tilings with infinite local complexity can also have transversals that are Cantor sets, but they can also have more complicated transversals. Lemma 3.2 of \cite{Me.Lorenzo.ILCfusion} states that having a totally disconnected transversal is a topological invariant of tiling spaces.

\begin{lem}\cite{Me.Lorenzo.ILCfusion}
If two tiling spaces are homeomorphic and one has a totally disconnected transversal, then so does the other.
\end{lem}

%In many cases the transversal can be constructed as an inverse limit of approximants $\transversaln$ under the `forgetful map' $f_n$, where $\transversaln$ is the set of all patches around the origin of size $n$ in the transversal.
The transversal of the pinwheel tiling looks like two Cantor sets, each crossed with a circle.  The way to see this is to first imagine a pinwheel tile with the control point at the origin.  The set of all tilings that contain this tile will be a Cantor set since distinct tilings are always separated by some amount determined by the closest place on which they differ, yet each tiling is the limit of a sequence of other tilings.  Now this Cantor set must be rotated in all amounts to get half the tiling space.  The other half of the space is obtained by doing the same thing with the flip of the pinwheel tile we started with.
We describe the nature of the transversal for several examples in Section \ref{example.analysis}.

\section{Ergodic theory applied to ILC tiling systems}\label{sec:ergodic}

Since tilings can be used to model the atomic structure of quasicrystals, the statistical, large-scale approach of ergodic theory makes sense: anything happening on a set of measure zero isn't physically observable and so can be ignored. {\em Miles of Tiles}\cite{Radin.miles} is an exposition of the method that explains the physical motivation for non-physicists.  We begin by interpreting fundamental dynamics concepts to our situation.

\subsection{Tiling dynamical systems, minimality, repetitivity, and expansivity}
Translation provides a natural action of $\R^d$ on $\Tsp$ that is continuous in the tiling metric and allows us to take a dynamical approach.
\begin{dfn}
A {\em tiling dynamical system} $(\Tsp, \R^d)$ is a tiling space $\Tsp$ along with the action of $\R^d$ by translation. 
\end{dfn}

A dynamical system is said to be {\em minimal} if the orbit of every tiling under translation is dense.  A minimal FLC tiling system has the property that all possible patches of any size can be found in any given tiling $\T$.   Since there are many more patches in an ILC system, minimality guarantees that every patch found in any tiling can be arbitrarily well approximated by one from any given tiling $\T$.   It is fairly easy to construct a minimal ILC tiling space by using traditional techniques, for instance with substitution as in Section \ref{sec:hierarchical}.

A tiling $\T$ is said to be {\em repetitive} if for every patch $P$ that appears in $\T$ and every $\epsilon >0$, there is an $R$ for which every ball of radius $R$ in $\T$ contains a patch that is within $\epsilon$ of $P$.    The orbit closure of $\T$ is a minimal tiling system if $\T$ is repetitive.

A tiling dynamical system is said to be {\em expansive} if there is a $\delta >0$ such that
whenever $d(\T-x, \T'-x) < \delta$ for all $x \in \R^d$, it means that $\T = \T' - y$ for some $y \in \R^d$ with $|y|<\delta$.  In an expansive system, then, the only way for the entire orbits of two tilings to be close is if they were small translates of one another to begin with.  FLC tiling spaces, like their cousins the shift spaces, always have expansive dynamical systems.   However, infinite local complexity brings us examples of tiling systems that do not have expansive dynamics.  Such an example appears as our example \ref{Solenoid.extensions}.

\subsection{The Borel topology and cylinder sets}
\label{Borel.topology.trim.sets}
In classical symbolic dynamics it is commonplace to consider the set of all sequences that have a specific symbol or word in a given location, and this set is called a cylinder set.   This notion generalizes nicely to the FLC tiling situation, where we need to specify both a patch $P$ and an open set  $U$, such that the cylinder set $\Tsp_{P,U}$ is the set of all tilings that contain the patch $P$ in a location designated by $U$.    Two properties of cylinder sets are essential to bring into the ILC situation.   First, they generate the metric topology.   Second, they can be used to compute the frequency with which the patch $P$ appears throughout the tiling space.

When we have infinitely many different two-tile patches, the cylinder sets based on single patches do not generate the topology.  Moreover, it is possible that every individual patch has frequency 0.  This means we need to make cylinder sets based on 
{\em sets} of patches, for instance, the set of all patches that are within $\epsilon$ of some particular patch.   To measure frequency accurately we need to define sets of patches that don't contain any `repeats'  up to translation:  
\begin{dfn}
A set of patches $I$ is said to be {\em trim} if, for some fixed open set $U \subset \R^d$ and every $\T \in \Tsp$, there is at most one patch $P \in I$ and point $x \in U$ for which $P-x \subset \T$.   \end{dfn}
\noindent Thus a trim set does not contain patches that are arbitrarily small translates of one another, or patches that sit in arbitrarily small translates of other patches.
%a trim set of patches does not contain two translates of the same patch, or any patch that is contained in the translate of any other.

\begin{dfn} Let $U \subset \R^d$ and let $I$ be a set of patches. The {\em cylinder set} $\Tsp_{I,U}$ is the set of all tilings in $\Tsp$ for which there is some patch $P \in I$ and point $x \in U$ for which $P-x \in \T$.  
\end{dfn}
 If $I$ is a trim set with a small enough $U$, we know that a tiling can only be in the cylinder set via one specific patch $P$ and point $x$.  If we let $\chi_{I,U}$ be the indicator function for this set, then $\chi_{I,U}(\T - x)$ as $x$ ranges through some subset of $\R^d$ will count the number of times a patch from $I$ appears in $\T$ in that subset, without overcounting. 

\begin{prop}
Cylinder sets given by trim sets generate the metric topology on $\Tsp$.
\end{prop}
\begin{proof}
We establish that every ball of radius $\epsilon$ around a tiling $\T$ can be obtained as a cylinder set.   Take the smallest patch in $\T$ that contains $B_{1/\epsilon}(0)$ and call it $P$, and denote by $x$ the control point of a tile in $P$ containing the origin.   The set of all patches that are within $\epsilon$ of $P$ can be partitioned into a trim set of translation classes $I$: take all patches $P'$ that have a control point at $x$ and for which $d(P, P') < \epsilon$.   
Then $\Tsp_{I, B_{1/\epsilon}(0)}$ is a cylinder set that equals the ball of radius $\epsilon$ around $\T$.
\end{proof}

\subsection{Translation-invariant measures and patch frequency}
\label{General.ILC.measures}

We begin this discussion by reviewing how translation-invariant Borel probability measures can be used to compute frequencies in the FLC case.   Given some finite patch $P$, if $U$ is a sufficiently small open set and $\mu$ is an invariant measure we can define the frequency of $P$ to be $freq_\mu(P) = \frac{\mu(\Tsp_{P,U} ) } {Vol(U)}$.  If $\mu$ is ergodic then by the ergodic theorem  for $\mu$-a.e $\T$ we have
$$freq_\mu(P) = \lim_{R \to \infty} \frac{1}{Vol(B_R(0))Vol(U)} \int_{B_R(0)} \chi_{P,U}(\T - x) dx, $$ where $\chi_{P,U}$ is the indicator function for $\Tsp_{P,U}$.  The integral represents the number of times we see a copy of $P$ in the ball of radius $R$ around the origin in $\T$, so averaging this by the size of the ball gives us the frequency of $P$.

In the ILC case, when $\mu$ is a translation-invariant measure and $I$ is a trim set we still see that $\mu(\Tsp_{I,U})$ is a multiple of $Vol(U)$ for all sufficiently small sets $U$.  Thus we can define the frequency of $I$ to be $freq_\mu(I) = \frac{\mu(\Tsp_{I,U})}{Vol(U)}$ as before.   And as before we are justified in the use of the word ``frequency" by the ergodic theorem.   If $I$ is a trim $\epsilon$-ball around some patch $P$, then $freq(I)$ is the percent of time we see patches that look almost exactly like $P$.

Let $\protoset_n$ be the set of all connected $n$-tile patches that have a control point at the origin and are translates of patches that appear in $\Tsp$. The metric on patches gives us a measurable structure on $\protoset_n$, and since every subset of $\protoset_n$ is trim, $freq_\mu$ forms a measure on $\protoset_n$.   If we want, we can consider $\protoset_\infty = \bigcup \protoset_n$, which is not itself a trim set.   However, since any subset of a trim set is trim, we can consider $freq_\mu$ to be a measure on (the set of measurable subsets of) any trim subset of $\protoset_\infty$.

If $\mu$ is a probability measure, then the frequency measure on $\protoset_n$ is {\em volume-normalized}, meaning that $\int_{\protoset_n} Vol(P) freq_\mu(dP) = 1$.   This follows from the fact that $\mu(\Tsp) = 1$ and  $\Tsp$ can be arbitrarily finely approximated by cylinder sets of the form $\Tsp_{I_n(j_n),U_n}$, where $I_n(j_n) = B_\epsilon(P_n(j_n))$ and $U_n = \supp(P_n(j_n))$ for some representative set of $n$-supertiles.

\subsection{Entropy and complexity} \label{entropy.complexity}

We develop a notion of complexity based on the standard form in symbolic dynamics, but taking ideas from topological pressure theory and the topological entropy of flows. The complexity function distinguishes the sort of infinite local complexity represented by the solenoid (example \ref{Solenoid.extensions})  from that of, say,  tilings which have a higher topological dimension than their ambient dimension (example \ref{DPVexample}, for instance).

There are three interrelated ways to define the complexity function, all of which yield slightly different actual numbers but have the same asymptotics and are based on the idea that complexity should count the number of patches of size $L$ one might see in $\Tsp$.  To that end we define a metric $d_L$ on $\Tsp$ for each $L > 0$ by
$$d_L(\T,\T') = \sup_{x \in [0,L]^d}\{d(\T-x,\T'-x)\}$$
Two tilings will be within $\epsilon$ of one another in this $d_L$ measure if their patches on $[-1/\epsilon, L + 1/\epsilon]^d$ are within $\epsilon$ in the patch metric.  Our complexity functions will count up how many such patches there are.

For any $\epsilon > 0$ and $L > 0$ we define $N_1(\epsilon, L)$ to be the minimum number of balls of $d_L$-radius $\epsilon$ it takes to cover $\Tsp$.   We define $N_2(\epsilon, L)$ to be the minimum number of sets of $d_L$-diameter $\epsilon$ it takes to cover $\Tsp$.  It is clear that since every open cover using balls of $d_L$-radius $\epsilon$ is a cover by sets of $d_L$-diameter $2 \epsilon$, we know that $N_2(2\epsilon,L) \le N_1(\epsilon,L)$.

Our third version of a complexity function relies on the idea of an {\em $\epsilon$-separated set:}   a set of tilings in $\Tsp$, no two of which are within $\epsilon$ of each other in the $d_L$ metric.  We define $N_3(\epsilon, L)$ to be the maximum cardinality of an $\epsilon$-separated set.  If we have such a set then we can cover $\Tsp$ with balls of $d_L$-radius $\epsilon$ centered on its elements, so we have that $N_3(\epsilon,L) \ge N_1(\epsilon,L)$.   Also, since any set of diameter $\epsilon$ can contain at most one element of an $\epsilon$-separated set, we have that $N_2(\epsilon,L) \ge N_3(\epsilon,L)$. Thus we have:
$$N_2(2 \epsilon,L)  \le N_1(\epsilon,L) \le N_3(\epsilon,L) \le N_2(\epsilon,L)$$

If we let $N$ denote any of these complexity functions, we can look at what happens as $\epsilon$ goes to $0$ and/or as $L \to \infty$.  For any given $L$ we see that even for tilings with finite local complexity $\lim_{\epsilon \to 0}N(\epsilon,L) = \infty$.    Instead we should fix an $\epsilon$ and investigate $\lim_{L \to \infty} N(\epsilon,L)$.      We say that $\Tsp$ has {\em bounded complexity} if $N(\epsilon, L)$ is bounded by some function of $\epsilon$, independent of $L$.  We say it has {\em polynomial complexity} if $N(\epsilon,L)$ is bounded by $C(\epsilon)(1+L)^\alpha$, where $C$ is some function of $\epsilon$ and $\alpha$ is some positive constant.

\begin{dfn}
The {\em $\epsilon$-entropy} of the tiling dynamical system $(\Tsp, \R^d)$ is given by 
$$h_\epsilon(\Tsp) = \limsup_{L \to \infty}  (log(N(\epsilon,L)))/ L^d.$$   If $\lim_{\epsilon \to 0} h_\epsilon(\Tsp) = h(\Tsp)$ is finite, then we say the system has finite entropy equal to $h(\Tsp)$.
\end{dfn}

The usual complexity function $c(n)$ for a one-dimensional symbolic sequence on a finite number of letters counts the number of distinct words of length $n$.   If we consider the sequence to be a tiling with labelled unit interval tiles, then any of our complexity functions $N(\epsilon,L)$ are approximately equal to $c([L + 2/\epsilon])/\epsilon$.

%
%There is also a dynamical definition of transversal that Ian knows and that I'm calling an {\em abstract transversal}.  The basic idea is that the transversal acts as a sort of global Poincare section for the dynamics of translation.  It is some closed subset of the space that hits any orbit in a Delone set, basically.   It is clear that the canonical transversal has this type, and I don't know where we go from here.

%
%There is a more nuanced definition of expansivity for our dynamical systems.   We say that $(\Tsp,\R^d)$ is {\em strongly expansive} if the induced action on $\transversal$ is expansive.   In the FLC case, expansivity implies strong expansivity, but there are ILC spaces that are expansive but not strongly expansive.   For instance

\section{Hierarchical tilings: substitution and fusion}
\label{sec:hierarchical}

An important theme in the study of aperiodic order is hierarchical structures: sequences or tilings that can be seen as possessing structure at arbitrarily large length scales.   The earliest work in this direction was on substitution sequences, which %have inspired a few generations of mathematicians and 
are surveyed in \cite{Fogg}.  Self-similar tilings were a natural generalization to the tiling situation, and have also been studied extensively in the FLC case (\cite{Baake-Grimm-book} is an excellent reference).   However, such hierarchical construction methods can lead naturally to tilings with infinite local complexity.  Early examples of tilings with infinite local complexity arose from tilings with a finite number of tile sizes and a substitution algorithm that forced the tiles to slide past one another in infinitely many ways \cite{Danzer, Kenyon.rigid}.   We have selected three examples that show some of the things that can happen when infinite local complexity arises in a hierarchical tiling.  

\subsection{Generating hierarchical tilings I: substitution}\label{substitution}
The earliest form of substitution was for symbolic systems, where there is some discrete alphabet $\aaa$ and some substitution rule $\sigma: \aaa \to \aaa^*$ that takes letters to words.   For instance, the Fibonacci substitution has $\aaa = \{a,b\}$, with $\sigma(a) = ab$ and $\sigma(b) = a$.   One can iterate the substitution by substituting each letter individually and concatenating the results.   In the Fibonacci example we have
$$\sigma^2(a) = \sigma(a)\sigma(b) = ab \, a \qquad \sigma^3(a) = \sigma(a) \sigma(b) \sigma(a) = ab \, a \, ab$$ and so on.   One can generate infinite sequences in this manner.

Extending this to the tiling case in one dimension is simple because tiles are intervals and can be concatenated without discrepancy.   However, once we are in two dimensions the tiles have geometry that can prevent the tiles from fitting together.  The first way around this was to devise inflate-and-subdivide rules that generate self-similar or self-affine tilings via linear expanding maps.   Many beautiful examples have been discovered and investigated, and can be found on 
the Tilings Encyclopedia website \cite{Tiling.encyclopedia}.

An inflate-and-subdivide rule requires a linear expansion map $\phi: \R^d \to \R^d$ such that for each prototile $t \in \ppp$, the expanded set $\phi(\supp(t))$ can be expressed as a union of tiles equivalent to prototiles from $\ppp$.  We write $S(t)$ to represent the patch of tiles that result from the inflate-and-subdivide process, called a $1$-supertile.   We can apply the substitution rule to the tiles in the patch $S(t)$ to obtain the patch $S^2(t)$, which we call a $2$-supertile.   Repeated substitution produces higher-order supertiles that grow to cover $\R^d$ in the limit.    %In the FLC case you should look at \cite{Sol.self.similar}.   Here is an example of a one-dimensional inflate-and-subdivide rule on an infinite prototile set.

\begin{ex}{\em Infinitely many tile lengths.} \label{Variable.size.example}
Let $[1,3]$ be both the label set and the set of tile lengths for a one-dimensional tiling as in our first introductory example.  For $x \in [1,3]$, the tile denoted $t_x$ is taken to be of type $x$ and $\supp(t_x)$ is an interval of length $x$.    The control points are taken to be the left endpoints.   We can take the metric on the label set to be given by  $d_\lcal(x,y) = |x - y|/2$, which is somewhat arbitrary but agrees with the Hausdorff distance of two tiles of lengths $x$ and $y$ that have the same midpoint.

We define a substitution rule that inflates by the expansion map $\phi(x) = 3x/2$ and subdivides the result only if it is larger than 3.  
If $x \in [1,2]$ we define $S(t_x) =  t_{3x/2} $, supported in the interval $\phi(\supp(t_x))$.   If $x \in (2,3]$ we define $S(t_x) = t_{x} \cup t_{x/2}$, again supported in the interval $\phi( \supp(t_x))$.   Notice that the substitution rule is discontinuous:  two tiles with lengths on either side of $2$ substitute to patches that are not close in the patch metric since they have different numbers of tiles.

In figure \ref{vss.its} we show a 9-supertile for the substitution, with the interval lengths coded by color [greyscale]; tiles close in color are close in length.   We see consecutive copies of the same tile on three occasions.
\begin{figure}[ht]
\includegraphics[width=6in]{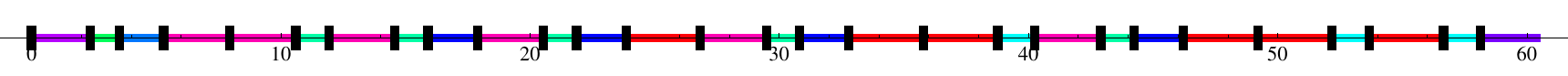}
\caption{Nine iterations of the substitution rule, applied to $\pi/2$.}
\label{vss.its}
\end{figure}

A version of the tiling space generated by this rule, considered from a fusion standpoint, is studied in \cite{Me.Lorenzo.ILCfusion}, where it is shown to be minimal and have a totally disconnected transversal and a unique translation-invariant Borel probability measure that is nonatomic.  We investigate more about this tiling space in Section \ref{vss.analysis}.

\end{ex}

The geometric rigidity imposed by the linear map $\phi$ can be loosened somewhat.  Tiling substitution rules exist such that any tile $t$ is substituted by a patch of tiles $S(t)$, but this patch may not be supported on a set that is a linear expansion of $t$.   These have been called combinatorial substitutions \cite{My.primer}, a special case of which is known by the term ``generalized substitutions" \cite{Gensubref}.

A straightforward way to generate tiling substitutions in $\R^d$ is to begin with the direct product of $d$ one-dimensional substitutions.   Given $d$ substitutions $\sigma_1, \sigma_2, ..., \sigma_d$ on alphabets $\aaa_1, \aaa_2,...\aaa_d$ we can define $$\sigma(a_1, a_2, ..., a_d) = \left(\sigma_1(a_1),  \sigma_2(a_2), ... \sigma_d(a_d)\right)   $$
A tile associated with the label $(a_1, a_2, ..., a_d)$ is a $d$-dimensional rectangle, the length of the $i$th side depending on $a_i \in \aaa_i$.   In a direct product tiling substitution the 
tiles must line up facet-to-facet and thus always have finite local complexity.  %Some work has been done on the tiling model of this \cite{Mozes,Livshits}.

These can be made into the more interesting {\em direct product variation (DPV)} substitutions, one of which is the example below.  To construct such a substitution we rearrange the inside of at least one of the substituted tiles in order to break the direct product structure.   Care must be taken to ensure that the rearranged interior still forms a legal patch when substituted so that the substitution admits tilings.   

Whether there is finite or infinite local complexity depends on combinatorial, number-theoretic, and/or geometric details.   One with ILC, based on the product of $a \to abbb, b \to a$ with itself, is the primary ILC example in \cite{Me.Robbie} and requires four tile sizes.   The simpler example we present here is similar to the one whose cohomology was computed in \cite{Me.Lorenzo}. 

\begin{ex} {\em Direct product variation (DPV).} 
\label{DPVexample}
Let $\sigma_1: a \to abbb, b \to a$ and $\sigma_2: c \to cc$. 
There are two rectangular tile types we call $A = a \times c$ and $B = b \times c$, where we think of $a, b,$ and $c$ as representing both intervals, their lengths, and their labels.    The direct product will then be
$$\begin{array}{|c|} \hline A \\ \hline \end{array}
\to
 \begin{array}{| c | c | c | c |}
\hline
A & B & B & B \\
\hline
A & B & B & B \\
\hline
\end{array}\,\,, \qquad \begin{array}{|c|} \hline B \\ \hline \end{array}
\to
 \begin{array}{| c | }
\hline
A \\
\hline
A  \\
\hline
\end{array}$$

We can vary this direct product as follows and be guaranteed that the substituted tiles will still fit together.
$$\begin{array}{|c|} \hline A \\ \hline \end{array}
\to
 \begin{array}{| c | c | c | c |}
\hline
A & B & B & B \\
\hline
B & B & B & A \\
\hline
\end{array}\,\,, \qquad \begin{array}{|c|} \hline B \\ \hline \end{array}
\to
 \begin{array}{| c | }
\hline
A \\
\hline
A  \\
\hline
\end{array}$$

By varying the widths of $A$ and $B$ we can obtain tilings with either finite or infinite local complexity.   If the widths are irrationally related,  the substitution rule admits tilings with ILC.
%
%\begin{figure}[ht]
%voila
%\caption{The substitution rule for a DPV}
%\label{DPVsubs}
%\end{figure}
Figure \ref{DPVits} shows three iterations of the $A$ tile  using the widths $a = (1 + \sqrt{13})/2, b= 1$, which are the natural widths for the self-affine tiling for this substitution.  Horizontal fault lines are beginning to develop, with mismatches between the tiles above the lines and those below.
Each iteration of the substitution produces new offsets along the fault lines, ultimately resulting in infinite local complexity.    The connection between these fault lines and the projection method will be discussed when we fully analyze this tiling in Section \ref{DPV.analysis}.

\begin{figure}[ht]
\includegraphics[width=4in]{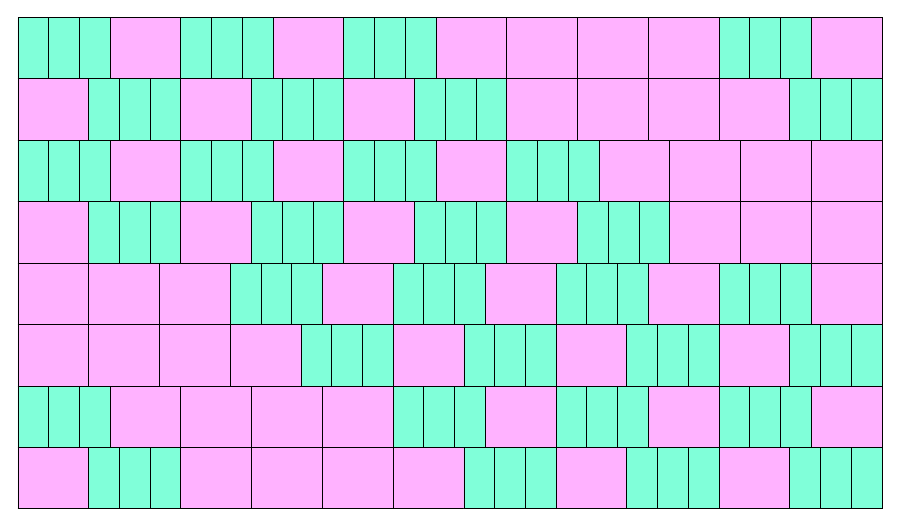}
\caption{A direct product variation $3$-supertile}
\label{DPVits}
\end{figure}

\end{ex}

\subsection{Generating hierarchical structures II: fusion}

Like substitution, fusion constructs tilings by making a series of $n$-supertiles that get larger and larger at each level.   The difference is that while substitution constructs an $n$-supertile by replacing each tile in an $(n-1)$-supertile with a substituted tile, fusion constructs an $n$-supertile by concatenating or `fusing' a number of $(n-1)$-supertiles.   We could think of substitution as being a cellular model:  each tile is a cell that can expand and subdivide itself into new cells.   Fusion is an atomic model:  each tile is an atom that can bond to other atoms to form molecules, which can themselves bond together to form larger structures.  We refer the reader to \cite{Me.Lorenzo.Fusion,Me.Lorenzo.ILCfusion} for technical details and more examples, but we describe many key points here.

The prototile set $\ppp_0$ will serve as our 0-supertiles.  Recall that there is a compact label set which we now call $\lcal_0$ that labels the prototiles and generates a tile and patch metric.   The $1$-supertiles are defined to be a set of finite patches $\ppp_1$ of tiles from $\ppp$.  We require that there be a compact $\lcal_1$ that labels the $1$-supertiles, so that we may write $\ppp_1 = \{P_1(c) \,\, |\,\, c \in \lcal_1\}$.    It is convenient but not necessary to require that if $c_n \to c$ in $\lcal_1$, then $P_1(c_n) \to P_1(c)$ in the patch metric generated by $\lcal_0$.  There are examples where the fusion and/or substitution is only piecewise continuous, for instance example \ref{Variable.size.example}.  

We make our set of $2$-supertiles $\ppp_2$ by requiring that each element of $\ppp_2$ be a {\em fusion} of $1$-supertiles: a finite, connected union of patches that overlap only on their boundaries.  We require that $\ppp_2$ is labelled by some compact label set $\lcal_2$, and we write $\ppp_2 = \{P_2(c)\,\, |\,\, c \in \lcal_2\}$. It is convenient if the patch metric generated by $\lcal_2$ is compatible with the patch metrics generated by $\lcal_1$ and $\lcal$ in the sense of the previous paragraph.

We continue in this fashion, constructing our $n$-supertiles as fusions of $(n-1)$-supertiles and requiring that each $\ppp_n$ be labelled by a compact set $\lcal_n$.    The {\em fusion rule} $\rrr$ is the set of all supertiles from all levels:
$$\rrr = \{\ppp_0, \ppp_1, \ppp_2, ...\}$$
We say a tiling $\T$ is {\em admitted} by $\rrr$ if every patch of tiles in $\T$ is equivalent to one appearing inside of a supertile from $\rrr$. The tiling space $\Tsp_\rrr$ is the set of all tilings admitted by $\rrr$, and $\rrr$ can be thought of as the language of $\Tsp_\rrr$.  In order for a fusion rule to admit an infinite tiling, the sizes of supertiles must be unbounded.  Moreover,  $\Tsp_\rrr$ is a translation-invariant tiling space and can be analyzed dynamically.   

In order to avoid trivialities, we assume that $\Tsp_\rrr$ is not empty and that every element of $\ppp_n$, for each $n$, appears somewhere inside an infinite tiling in $\Tsp_\rrr$.    We can assume the latter without loss of generality since any superflous supertiles can be removed from $\rrr$ without changing $\Tsp_\rrr$.

\begin{dfn}
An {\em infinite-order supertile} $P_\infty$ is a tiling of an unbounded region of $\R^d$ for which there is a sequence of supertiles $P_n \subset \ppp_n$ and translations $x_n \in \R^d$ for which $P_\infty=\lim_{n \to \infty}(P_n - x_n)$ and $P_n - x_n \subset P_{n+1} - x_{n+1}$ for all $n$.
\end{dfn}

Tilings admitted by a FLC fusion rule $\rrr$ are either one or the concatenation of finitely many infinite-order supertiles.  In the ILC situation there is another possibility, that they are the limit of such infinite-order supertiles in the big ball metric.

\begin{ex} {\em Solenoid extensions\footnote{These are closely related to discrete actions known as  ``Toeplitz flows" and are surveyed in \cite{Toeplitz}.}.} 
\label{Solenoid.extensions} We present a simple family of fusions that are not  substitutions.  All of the examples are measurably conjugate to the dyadic solenoid system, which is described as an inverse limit in chapter \cite{Lorenzos.chapter} of this volume and can be seen as a height-1 suspension of the dyadic odometer.   However, the topology  is highly sensitive to changes in the prototile set.

For this family  of tilings, the support of all prototiles is $[0,1]$.  The label set $\lcal$ is a compactification of the non-negative integers $\N_0$ and we write $\lcal = \N_0 \cup \lcal_c$.  We denote a tile of type $l$ as $A_l$.  Regardless of the specific nature of $\lcal$ the fusion rule will be constructed as follows.
%We begin with a set of tiles $A_1, A_2, ...$, where the support of each tile is a unit interval and the label set is countably infinite.  We assume that the label set contains additional labels that compactify it, in which case we have limit tiles that are unit intervals labelled by the limit points.  We can alter the topological properties of the tiling space immensely both by changing the topology of the compactifying set and also by changing which sequences converge to what.  We will consider only the most trivial case here, a one-point compactification.
%For this example, let $A_n = ([-.5,.5], 1/n)$ if $n \in \Z^+$ and let $A_0 = ([-.5,.5],0)$ with the usual Euclidean distance on the label set $\lcal = \{0, 1, 1/2, 1/3, ...\}$.   The prototile set is therefore the set $\{A_0, A_1, ...\}$, unit length tiles with control points at the origin, which is obviously compact with $A_0$ being the only limit tile.
Letting the set of 0-supertiles, $\ppp_0$ be the prototile set, we construct our set of 1-supertiles as follows.   For $l \in \lcal_c$ and $l = 1,2,3,4,...$ we define $P_1(l) = A_l \cup (A_0 + 1)$, that is, the two-tile patch supported on $[0,2]$ given by the concatenation of $A_l$ and $A_0$.   We hope it does not risk too much confusion about the precise support of $P_1(l)$ if we abuse notation and write $P_1(l) = A_l A_0$. 
We write $\ppp_1 = \{P_1(l), l \in \lcal_1\}$, where $\lcal_1$ is the subset of $\lcal$ given by  $\{1,2,3,4,...\}\cup \lcal_c$.    Now to generate the 2-supertiles we concatenate each 1-supertile with label in $\{2,3, 4, ...\} \cup \lcal_c$ with $P_1(1)$, so that we have $P_2(l) = P_1(l)P_1(1) = A_l A_0 A_1 A_0$.   (Again we abuse notation but we know this is supported on $[0,4]$.)   The set of 2-supertiles takes the form $\ppp_2 = \{P_2(l), l \in \lcal_2\}$, where $\lcal_2$ is the subset of $\lcal$ given by  $\{2,3,4,...\}\cup \lcal_c$.    Similarly the set of 3-supertiles will have the form $P_3(l)= P_2(l) P_2(2) = A_lA_0 A_1 A_0 A_2 A_0 A_1 A_0$, for $l \in \lcal_3$.
The general form for the set of $k$-supertiles, $k = 1, 2, 3, ... $ is
$$ \ppp_k = \{ P_{k-1}(l) P_{k-1}(k-1), l \in \lcal_k\}.$$

By looking at the form for $P_3(l)$, we see that there will be an $A_0$ in every other slot, an $A_1$ in every fourth slot, and can surmise that there will be an $A_n$ in every $2^{n+1}$th slot.   In fact we can generate an infinite tiling admitted by this fusion rule by a method quite similar to the construction of a Toeplitz sequence.  We begin by placing infinitely many $A_0$'s on the line with a unit space between them.   Of the remaining spaces, we alternate by filling one with $A_1$ and leaving the next one empty.   We continue in this fashion, filling every other of the remaining spaces with an $A_2$, and so on.   When the process is finished, there may or may not be one empty space.   If there is, it should be filled with an $A_l$ with $l \in \lcal_c$.

Whatever form $\lcal_c$ takes, the tiling space admitted by this fusion rule has infinite local complexity.  When $\lcal_c$ consists of a single point we will see in Section \ref{Solenoid.Complexity} it is actually less `complex' than FLC examples because its complexity is bounded.  When $\lcal_c$ is finite but greater than 1 we can use Theorem \ref{One-D-FLCresult} to show that it is topologically conjugate to a tiling space with finite local complexity.   We will look closely a few special cases in Section \ref{Solenoid.analysis}.

\end{ex}

Tilings generated by substitution can always be seen as being generated by fusion since a tile that has been substituted $n$ times can be seen as the concatenation of tiles that have been substituted $(n-1)$ times.  However the converse is not true.   Fusion is more general and can allow us to vary the size of the supertile sets from one level to the next.   They also can allow us to vary the fusion patterns from level to level, and can account for generalized substitutions and random substitutions.

\subsection{Transition matrices and transition maps}

\subsubsection{Transition matrices}
When there is an inflate-and-subdivide rule or a substitution rule for a tiling with a finite number of tile types it is very handy to compute the {\em transition matrix $\M$} of the substitution.   The $(i,j)$th entry of the matrix is given by the number of tiles of type $i$ in the substitution of the tile of type $j$.    The matrix $\M^n$ knows how many of each prototile type can be found in the $n$-supertiles of each type.  We say $\M$ is {\em primitive} if there is some $n$ for which $\M^n$ has all positive entries.   This means that each $n$-supertile contains copies of every tile type.

 In the planar case, when the inflate-and-subdivide rule is a similarity, we can consider the expansion to be by some complex number $\lambda$, which is the Perron eigenvalue of the transition matrix.
Early in the study of self-similar tiling dynamical systems it was discovered that the algebraic type of $\lambda$ had a significant impact on the dynamics.  Working in the context of finite local complexity,  Solomyak \cite{Sol.self.similar} showed that under the conditions of primitivity and recognizability a self-similar tiling of the line fails to be weakly mixing if and only if $|\lambda|$ is a Pisot number (an algebraic integer, all of whose algebraic conjugates are less that one in modulus).   In the same work he showed that a tiling of the (complex) plane fails to be weakly mixing if and only if $\lambda$ is a complex Pisot number.

It turns out that the algebraic type of the expansion constant also has an effect on local complexity.
In \cite{Me.Robbie} it is shown that under common conditions, if the length expansion is a Pisot number then tilings admitted by the substitution must have finite local complexity.   In this situation it is known that there must be some measurable spectrum and cannot be weakly mixing \cite{Sol.self.similar}.  Thus weak mixing and local complexity are linked via the expansion constant; when it isn't Pisot there is a chance of the local complexity becoming infinite.

If a  fusion rule has a finite number of $n$-supertiles at each stage, we only need to generalize the idea of a transition matrix for substitutions slightly.   We now have transition matrices $\M_{n,N}$ whose (i,j)th entry represents the number of $n$-supertiles of type $i$ in $P_N(j)$.    For any $m$ between $n$ and $N$ we have $\M_{n,N} = \M_{n,m}\M_{m,N}$.   Thus even though we are unable to use Perron-Frobenius theory, we are able to use a parallel type of analysis to determine the possible invariant measures and patch frequencies.    If the tiling space has ILC as in example \ref{DPVexample}, the patches that only appear along infinite fault lines have frequency 0 and all other patches have nonzero frequency \cite{Me.Lorenzo.ILCfusion}.

\subsubsection{Transition maps}

When we have a substitution or fusion on infinitely many prototiles or supertiles, the transition between levels must be a map
$\M_{n,N}: \ppp_n \times \ppp_N \to \Z$ for which $\M_{n,N}(P,Q)$ represents the number of $n$-supertiles of type $P$ in the $N$-supertile of type $Q$.    For fixed $Q$, there are only finitely many nonzero entries, but a fixed $P \in \ppp_n$ might appear in infinitely many tiles.   For instance, every tile type in the solenoid example has this property except ones with label in the compactification $\lcal_c$.

Even though the transition maps are no longer matrices, they can still be used to obtain the possible invariant measures for a large class of ILC fusions (see Section \ref{results.section}).  The overarching principle is that we have a measure $\rho_n$ on each supertile set $\ppp_n$ so that for any measurable, trim subset $I \subset \ppp_n$, $\rho_n(I)$ represents the frequency of seeing any supertile from $I$ in a tiling.   When the sequence of such frequencies $\{\rho_n\}$ behave nicely with respect to transition we obtain both a translation-invariant measure on the tiling space and a handy formula for computing the frequencies of all types of patches, not just supertiles.

\subsection{Recognizable, van Hove, and primitive fusion rules}
The definition of a fusion rule is sufficiently general as to encompass all tilings whatsoever, and therefore we need to put some restrictions on the rules to make them meaningful.   The three standard assumptions are that the fusion rule be recognizable, van Hove, and primitive.   Recognizability and the van Hove property are defined for fusions the same way whether the local complexity is finite or infinite, but primitivity requires a more subtle definition in the case of infinite local complexity.

Recognizability means that the substitution or fusion can be undone in a well-defined way.   For self-similar tilings, recognizability means that there is some finite `recognizability radius' $R$ such that if $\T$ and $\T'$ have identical patches in the ball $B_R(x)$, then they have the same substituted tile at $x$, situated in precisely the same way.   For fusion rules, we need a recognizability radius for each level of supertiles.    That is, for each $n \ge 1$ there is an $R_n > 0$ such that if $\T$ and $\T'$ have the same patch of $(n-1)$-supertiles  in $B_{R_n}(x)$, then they have the same $n$-supertile at $x$ in precisely the same location and orientation.  

If $\Tsp$ is either a substitution or fusion tiling space, we can define spaces $\Tsp^n$ to be the space of tilings from $\Tsp$ with the  $n$-supertiles considered to be the set of prototiles by `forgetting' all the tiles in their interiors.    Since every tiling in $\Tsp$ is a union of $n$ supertiles for any $n$, this is a well-defined tiling space.   There is always a map from $\Tsp^n$ to $\Tsp^{n-1}$ since we know how each $n$-supertile is constructed from $(n-1)$-supertiles.  However, this map is not necessarily invertible:  there could be tilings in $\Tsp^{n-1}$ that could be composed into tilings of $n$-supertiles in more than one way.  The substitution or fusion rule is recognizable if that is not the case and the map is a homeomorphism for each $n$.   %It is shown in \cite{Me.Lorenzo.Fusion} that every fusion rule is a factor of a recognizable fusion rule.

It is convenient to work with fusion rules for which all of the supertiles grow in area in a reasonable way, for instance without becoming arbitrarily long and skinny.   One way to avoid this is to require that the boundaries of supertiles are small relative to their interiors.   To this end, for $r > 0$ and any set $U \in \R^d$ we define $(\partial(U))^{+r}$ to be the set of all points in $\R^d$ that are within $r$ of the boundary of $U$. 
A sequence of sets $\{U_n\}$ in $\R^d$ is called {\em van Hove} if for every $r \ge 0$  we have that $\lim_{n \to \infty} \frac{Vol(\partial(U_n)^{+r})}{Vol(U_n)} = 0$.
A fusion rule is van Hove if any sequence of $n$-supertiles $\{P_n\}$, where $P_n \in \ppp_n$, is supported on a van Hove sequence.   This property is sufficient to ensure that the tiling space is not empty, but it is not necessary.    

The general idea behind primitivity is that given any $n$, there should be some $N$ for which every $N$-supertile contains $n$-supertiles of each type.    This definition makes sense when the sets of supertiles are finite.    When they are not, we simply require that for any $n$ and any open set $I$ of $n$-supertiles, there is an $N$ such that every $N$-supertile contains an $n$-supertile from $I$.  Thus, in the ILC case is that the size of $I$ affects the size of $N$, whereas in the FLC case a single $N$ can be chosen for all $n$-supertiles.     A primitive fusion or substitution has the property that for sufficiently large $N$, $\M_{n,N}(I,Q) \neq 0$ for all $Q \in \ppp_N$.

\section{Results about ILC tilings}
\label{results.section}

\subsection{Fault lines and fractured tiling spaces}
\label{Fault.and.fractured}

We have seen that it is easy to construct tilings of the plane with infinite local complexity.  An important question to ask is, what are the ways a planar tiling can have ILC?  When the prototile set is finite, the answer is given in \cite{Kenyon.rigid}. Up to translation it can only happen if there are  arbitrarily long line segments composed of tile edges where the tiles meet up in arbitrarily many different ways.   

\begin{thm}\cite{Kenyon.rigid}  \label{Kenyon.thm} A tiling of the plane with translated copies of a finite set of tiles either has only a finite number of local configurations or else contains arbitrarily long line segments in the boundaries of the tiles.
\end{thm}

Kenyon also shows in \cite{Kenyon.rigid} that if we allow infinitely many rotations of our finite prototile set  there is only one additional way that infinite local complexity can arise: Some of the tile edges would have to have circular arcs, so that a circular patch of tiles can be constructed.  This patch of tiles could then be rotated by arbitrary amounts inside any patch of tiles that surrounds it to create infinitely many different patches.

Kenyon called an infinite line of tile boundaries along which tiles can slide an {\em earthquake} and elsewhere in the literature it is often called a {\em fault line}.  Unfortunately fault lines do not have a unified definition, so we provide one here.
\begin{dfn}
A tiling $\T$ in a tiling space $\Tsp$ is said to have a {\em fault line} $\ell$ if there are infinitely many nonequivalent tilings $\T'$ such that $\T' = \T$ on one side of $\ell$ and $\T' = \T - x$ on the other side of $\ell$ for some $x \in \R^2$ that is parallel to $\ell$.
\end{dfn}
 However, fault lines are not topologically invariant:  one can take a tiling space containing fault lines and relabel all the tiles by `collaring' (see \cite{Lorenzos.chapter}): each tile is labelled by its corona.  The resulting tiling space will not contain fault lines per se, but they will still be {\em fractured}.    We offer here a new definition that is not specific to planar tilings and may be related to the proximal and asymptotic structure of tiling spaces (see \cite{Barge-Olimb}).
\begin{dfn}
A space $\Tsp$ of tilings of $\R^d$ has a {\em fracture in the direction of $x \in \R^d$} if there exists some $y \in \R^d$ and two tilings $\T, \T' \in \Tsp$ such that
$$\lim_{t \to \infty} d(\T -ty,\T' - ty) = 0 \qquad \text{ and } \qquad \lim_{t \to -\infty} d(\T -ty,\T' -x - ty) = 0$$
\end{dfn}

So a tiling space is fractured in the direction $x$ if there are tilings that asymptotically agree in the $y$ direction and, after an offset, asymptotically agree in the $-y$ direction.   There can be a large region `in the $x$ direction' in the middle of the tiling on which they do not agree.   In one dimension, a tile near the origin could be added, removed or resized.   In two dimensions, tilings with fault lines are fractured, and so are any tilings that are MLD to them.  Moreover tilings that are asymptotically proximal will also have fractured tiling spaces.   As an example, consider a chair tiling with an infinite diagonal of chairs that can be completely flipped without altering the rest of the tiling. 

Notice that $y$ is not uniquely defined.   It obviously can be rescaled, but in two or higher dimensions the direction can be changed.  The direction of $x$, however, cannot be changed without changing the direction of the fracture.   A tiling space can have multiple fractures in different directions:  see \cite{My.primer} for a planar tiling with translationally-finite prototile set that has fault lines in three independent directions.

Fractures and fault lines play an important role in the spectrum of tiling dynamical systems. Recall the following definition from higher-dimensional dynamics.
\begin{dfn}
The dynamical system $(\Tsp, \R^d)$ has an eigenfunction $f: \Tsp \to \C$ with eigenvalue $\alpha \in \R^d$ if for any $\T \in \Tsp$ and $y \in \R^d$, $$ f(\T - y) = \exp(2\pi i \alpha \cdot y) f(\T)$$
\end{dfn}
%Our discussion centers on continuous eigenfunctions rather than measurable ones, although in many cases of note (for instance self-similar tilings and substitutive sequences) the categories coincide.

\begin{thm}\footnote{This result is joint with L. Sadun.} \label{Spectrum.theorem}If $f \in C[\Tsp, \C]$ is a continuous eigenfunction with eigenvalue $\alpha \in \R^d$ and $\Tsp$ has a fracture in the direction of $x$, then $\alpha \cdot x$ is an integer.
\end{thm}

\begin{proof} Suppose that $y \in \R^d$ satisfies the fracture definition for tilings $\T, \T' \in \Tsp$. Since $f$ is uniformly continuous,
$0 = \lim_{t \to \infty} (f(T-ty) - f(T'-ty))
= \lim \exp(2\pi it \alpha \cdot y)[f(T)-f(T')]$,
so $f(T) = f(T')$. Similarly, taking a limit as $t \to -\infty$, we get $f(T)=f(T'-x)$. But then $f(T')=f(T'-x) = \exp(2 \pi i\alpha \cdot x) f(T')$, so
$\exp(2 \pi i \alpha \cdot x)=1$ and the result is proved.
\end{proof}

\begin{cor} \label{Spectrum.corollary} If the tiling space contains fractures in the direction of $kx$ for multiple values of $k \in \R$, then  $k \alpha \cdot x \in \Z$ for each of them. If any of the $k$'s are irrationally related, or if they can be arbitrarily small,  then $\alpha \cdot x = 0$. If the $k$'s have a greatest common factor $m$, then $ \alpha \cdot x \in  \frac{1}{m}\Z$.
\end{cor}

We will use this corollary to compute the topological spectrum of certain direct product variation substitutions in Section \ref{DPV.analysis}.

\subsection{General results for ILC fusion tilings}
If a fusion or substitution rule is primitive, then the tiling or sequence system associated with it is necessarily minimal regardless of local complexity \cite{Me.Lorenzo.Fusion,Me.Lorenzo.ILCfusion}.
It is possible for a fusion tiling space or substitution sequence space to be minimal even if it is not primitive, the Chacon substitution and its DPV analogues being examples.   If a fusion is both recognizable and van Hove, primitivity and minimality are equivalent:

\begin{thm}\cite{Me.Lorenzo.ILCfusion} \label{fusion.minimal.primitive} If the fusion rule $\rrr$ is primitive, then the fusion tiling dynamical system $(\Tsp_\rrr, \R^d)$ 
is minimal.   Conversely, a recognizable, van Hove fusion rule that is not primitive cannot have a minimal dynamical system.
\end{thm}

%It is interesting to consider that this theorem implies that the variable size substitution example from \ref{Variable.size.example}
%must therefore have a minimal dynamical system.   
 
 In Section \ref{Borel.topology.trim.sets} we defined a trim set of patches (basically, one that contains no repeats up to translation), and we explained in Section \ref{General.ILC.measures}  why it is appropriate to consider $freq_\mu(P) = \frac{\mu(\Tsp_{I,U})}{Vol(U)}$ to be the frequency of occurrence of patches from $I$ throughout $\Tsp$ (basically, the ergodic theorem).   For fusion tilings, we can say more about frequencies of patches by understanding frequencies of supertiles.    While our discussion is written in the context of ILC fusions, we note that the construction can be simplified to apply to the FLC case.
 
 Throughout this discussion it is essential that our fusion rule be recognizable and van Hove, and we refer the reader to \cite{Me.Lorenzo.ILCfusion} for proofs and more details.    Consider a set $I \subset \ppp_n$ of $n$-supertiles.  Supposing that we have chosen the control points of our $n$-supertiles so that each supertile contains some $\epsilon$-ball around the origin,  $I$ is automatically a trim set since by recognizability a tiling cannot have more than one supertile at any given location.  In a slight abuse of notation we define $\Tsp_{I,U}$ to be the set of all tilings that contain an $n$-supertile from $I$ at a location determined by $U$\footnote{The set $I$ of $n$-supertiles corresponds, by recognizability, to a trim set of patches of ordinary tiles, each of which is larger than its supertile by some recognizability radius.}.  
 
 We define a measure $\rho_n$ on $\ppp_n$ defined by $\rho_n(I) = \frac{\mu(\Tsp_{I,U})}{Vol(U)}$.   When $\mu$ is a probability measure, each $\rho_n$ is {\em volume normalized} in that
 $\int_{P \in \ppp_n} Vol(P) d\rho_n = 1$.   This can be seen as follows.  Suppose that $I_1, I_2, ..., I_j$ is a partition of $\ppp_n$ into sets of very small diameter and that $P_k \in I_k$ for all $k = 1, ..., j$.  Then $  \bigcup_{P_k }\Tsp_{P_k,Vol(P_k)}\approx \Tsp$ and so $\sum_{k= 1 }^j \rho_n(P_k)Vol(P_k)\approx \mu(\Tsp) =1 $.
 
The transition map $\M_{n,N}$ can be thought of as inducing a map from measures on $\ppp_N$ to measures on $\ppp_n$. 
For a fixed trim set $I \subset \ppp_n$, the function $A_{n,N}(I,Q)$ can be defined as the number of times an $n$-supertile from $I$ is contained in the $N$-supertile $Q$.  Since each such $Q$ contains only finitely many $n$-supertiles this function takes values in the nonnegative integers.

  Let $n <N$ be fixed, let $\nu_N$ be a measure on $\ppp_N$ and let $I \subset \ppp_n$ be a trim set.   We define a measure $\nu_n$ on $\ppp_n$ as
$$\nu_n(I) = (\M_{n,N} \nu_N)(I) = \int_{Q \in \ppp_N} \M_{n,N}(I,Q) d\nu_N$$
We say that  a sequence of measures $\{\nu_n\}_{n = 0}^\infty$ is {\em transition-consistent} if whenever $n < N$, $\nu_n = \M_{n,N} \nu_N$.  This means that the frequencies that $\nu_n$ assigns to $n$-supertiles is consistent with the frequencies that $\nu_N$ assigns to $N$-supertiles, for any $n \le N$.   
 
The fact that translation-invariant probability measures give rise to sequences of volume normalized and transition consistent supertile measures and vice versa is the subject of the next theorem. 

\begin{thm}\label{measures_are_sequences} \cite{Me.Lorenzo.ILCfusion}
  Let $\rrr$ be a 
  fusion rule that is van Hove and
  recognizable.  Each translation-invariant Borel probability measure
  $\mu$ on $\Omega_\rrr$ gives rise to a sequence of volume normalized
and transition consistent measures $\{\rho_n\}$ on $\ppp_n$.
Moreover, for any  trim set of patches $I$ 
\begin{equation} \label{measure_from_frequency}
freq_{\mu}(I) = \lim_{n \to \infty} \int_{P \in \ppp_n} \#( I \text{ in } P) d\rho_n,
\end{equation}
where $\# ( I \text{ in } P)$ denotes the number of translates 
of patches in the family $I$ that are subsets of $P$. 
Conversely, each sequence $\{\rho_n\}$ of volume normalized and transition consistent
measures 
corresponds to exactly one invariant measure $\mu$ via equation (\ref{measure_from_frequency}).  
\end{thm}

This means that the invariant measures for a recognizable van Hove fusion system are almost completely determined by the transition maps, a fact that is consistent with substitution sequence and self-similar tiling theory.   In those cases one looks at the transition matrix and uses the Perron-Frobenius theorem to conclude when the system is uniquely ergodic.   In the FLC fusion case the transition maps are always matrices and we can parameterize the set of all possible measures by a Choquet simplex related to the
matrix system.

 \subsection{ILC fusion tilings with finitely many $n$-supertiles}
When there are finitely many prototiles and finitely many $n$-supertiles at every level, then any set of $n$-supertiles contains only finitely many patches of a given size.   These patches are said to be {\em literally admitted} by the fusion rule.  If the tiling space has infinite local complexity there can be patches that are not literally admitted; we call these {\em admitted in the limit}.    An easy corollary to Theorem \ref{measures_are_sequences} is that the patches that are admitted in the limit have frequency 0.   More precisely, if we take any trim set $I$ of patches that are admitted in the limit, we clearly have that $\#(I \in P) = 0$ for any supertile $P$, so $freq_\mu(I) = 0$ by equation \ref{measure_from_frequency}.

In fact, the set of literally admitted patches is countable, and they support the frequency measure.
Thus the frequency measure is atomic and for any trim set of patches $I$,
$freq_\mu(I) = \sum_{P \in I} freq_\mu(P)$.  This is Theorem 4.4 of \cite{Me.Lorenzo.ILCfusion}.
 
Often in the planar situation we can prove topological weak mixing or obtain strong restrictions on the topological spectrum by combining Theorem \ref{Kenyon.thm} and Corollary \ref{Spectrum.corollary}.   The infinite local complexity comes from fault lines that in many examples have arbitrarily small shears.   If this happens in two independent directions the tiling space in question is topologically weakly mixing.
 
% \subsection{ILC fusion tilings with continuous supertile sets}
 %??
 
 \subsection{Homeomorphism of ILC and FLC tiling spaces}
 Every non-periodic FLC tiling space is homeomorphic, or even topologically conjugate, to an ILC tiling space obtained by collaring with infinite collars.  It is often possible to obtain an ILC tiling space from an FLC one geometrically, too.  That makes it important to know whether a given space with infinite local complexity can be converted into one with finite local complexity.   This classification problem is open in all dimensions except dimension one.

% \subsubsection{A positive result in one dimension}
  In one dimension it is possible to detect when a tiling space with ILC is homeomorphic to a tiling space with FLC.  It is necessary but not sufficient for the tiling space to have a totally disconnected transversal and expansive dynamics since all tilings with FLC have those properties.  However, it is possible to construct examples of ILC tilings that have these properties but cannot be homeomorphic to FLC tiling spaces, for instance by deforming the tiles of a solenoid extension \cite{Me.Lorenzo.ILCfusion}.  The property a one-dimensional tiling space $(\Tsp, \R)$ needs isn't expansivity, it is {\em strong expansivity}: %is {\em strongly expansive} if 
 the first return map to the transversal should be expansive as a $\Z$-action.

% This turns out to be essential for a one-dimensional tiling space to be homeomorphic to one with finite local complexity.  The stretched solenoid extension referred to above turns out to be expansive but not strongly expansive.
 
\begin{thm} \label{One-D-FLCresult}\cite{Me.Lorenzo.ILCfusion}
 If a one-dimensional tiling space is strongly expansive and the canonical transversal is totally disconnected, then it is homeomorphic to a one-dimensional tiling space with finite local complexity.
\end{thm}

There is a more general definition of strong expansivity for tilings of $\R^d$, but it seems unlikely that the $d$-dimensional analogue of Theorem \ref{One-D-FLCresult} holds.   A potential counterexample that is a variant of the pinwheel tilings appears in \cite{Me.Lorenzo.ILCfusion}.   It has uncountably many ergodic measures and therefore is probably not homeomorphic to a tiling space with finite local complexity.

\section{Analysis of our three main examples.}
\label{example.analysis}

\subsection{Example \ref{Variable.size.example}: variable size tile lengths}
Consider $\Tsp$ to be the tiling space admitted by the substitution in example \ref{Variable.size.example}.    Recall that the tile labels and lengths are in $[1,3]$, the expansion factor is $3/2$, and an expanded tile is decomposed into two tiles of lengths $1/3$ and $2/3$ of the expanded tile if it is larger than $3$ and otherwise not subdivided.

\label{vss.analysis}
\subsubsection{Minimality}
%We compute incorrectly:
%$$\M_{n,n+1}(x,y) = \left \{\begin{array}{ccccc}1 & \text{ if } x = 2y/3 \text{ and } x \le 2(3/2)^n \\
%1 & \text{ if } x= y/3 \text{ and } x > 2(3/2)^n \\
%1 & \text{ if } x = 2y/3 \text{ and } x > 2(3/2)^n \\
%0 & \text{ otherwise }
%\end{array} \right\}
%$$
The tiles in any tiling $\T$ all have lengths of the form $(2^j/3^k) y$, where $y \in [1,3]$ and $j$ and $k$ are nonnegative integers.   These lengths are dense in $[1,3]$ for any fixed $y$, so we know that it is possible that the orbit of $\T$ is dense in $\Tsp$.    We prove it by establishing primitivity.

The $N$-supertile of length $L$ is made up of $n$-supertiles of lengths $(2^j/3^k)L$, for nonnegative integers $j \le k$ in a range of values that depends more on $N$ than it does on $L$.   The largest value of $k$, denoted $\bar{k}$,  comes from the rightmost $n$-supertile of $L$ and has the property that $(2/3)^{\bar{k}} L \in [(3/2)^n, 3(3/2)^n]$.   The smallest value of $k$, denoted $\underline{k}$, comes from the leftmost $n$-supertile and has the property that $(1/3)^{\underline{k}} L \in [(3/2)^n, 3(3/2)^n]$.   All values of $k$ between these two values occur, and all values of $j$ for which $(2^j/3^k) L$ lie in range also appear, with the $j$s being consecutive integers in the range.   

Consider a set of $n$-supertiles of diameter $\epsilon$, so that all of the lengths are within $\epsilon$ of each other.   We can choose $N$ such that for any length $L$, the maximum of $k$ is sufficiently large so that every interval of size $\epsilon$ contains enough points of the form $(m /3^{\bar{k}})L$ that at least one of these points must be of the form $(2^j/3^k) L$, where $\underline{k} \le k \le \bar{k}$.   This proves the substitution is primitive.   Theorem \ref{fusion.minimal.primitive} implies that the variable size substitution example must therefore have a minimal dynamical system.   

\subsubsection{Weak mixing}  This system has no measurable or topological eigenvalues, as shown in this clever proof from L. Sadun.  Let $\Tsp$ be the tiling space from Example \ref{Variable.size.example} and let $\Tsp_\lambda$ be the tiling space obtained by expanding every tiling in $\Tsp$ by the linear map $f(x) = \lambda x$ for some $\lambda > 1$.  Then eigenvalues of $\Tsp_\lambda$ are exactly the eigenvalues of $\Tsp$ multiplied by $1/\lambda$.   

On the other hand we can consider the tiles in $\Tsp_\lambda$, which live in $[\lambda, 3\lambda]$, to be lengths for supertiles in $\Tsp$, and apply the decomposition map for supertiles to each tiling in $\Tsp_\lambda$.   This will result in tilings from $\Tsp$ and we have a map from $\Tsp_\lambda$ to $\Tsp$ that is a bijection everywhere except on the set of measure 0 for which the subdivision rule is discontinuous.  This means that $\Tsp_\lambda$ and $\Tsp$ are measurably conjugate and therefore have the same eigenvalues.% Thus the measurable spectrum of $\Tsp_\lambda$ will equal that of $\Tsp$ divided by $\lambda$.   

These two facts together mean that the set of eigenvalues of $\Tsp$ (i.e., its spectrum) is invariant by scaling by any $\lambda > 1$.   This means that the spectrum must either be  $\R$, the nonnegative reals,  the nonpositive reals, or 0.   However, we know that since $\Tsp$ is separable the spectrum must be countable.   Thus the only possible measurable eigenvalue is $0$.    Since continuous eigenfunctions are measurable this also implies topological weak mixing.

\subsubsection{Invariant measure}  The analysis of the related fusion rule in \cite{Me.Lorenzo.ILCfusion} can be adapted to find the sequence $\{\rho_n\}$ of transition-consistent and volume-normalized measures on $\ppp_n $.  We abuse notation and consider the set of $n$-supertiles to be the interval $[(3/2)^n, 3 (3/2)^n]$, in which case the measures of sets of supertiles is given by $\rho_n([a,b]) = \int_a^b f_n(x) dx$, where $dx$ is ordinary Lebesgue measure.   The only choice allowing a transition-consistent and volume-normalized sequence is
$$f_n(x) = \left\{ \begin{array}{cc} \frac{1}{ (3 \ln 3 - 2 \ln 2) x^2} & (3/2)^n \le x \le 2(3/2)^n \\  \frac{3}{ (3 \ln 3 - 2 \ln 2) x^2}& 2(3/2)^n < x \le 3(3/2)^n  \end{array} \right. $$    Although it is not obvious, it is true that the system is uniquely ergodic.

The invariant measure is invariant under scaling in the sense that $ \mu(\Omega_{I,U}) = \mu(\Omega_{\lambda I, \lambda U})$ for a trim set $I$ and denoting by $\lambda I$ the set of patches obtained by rescaling and subdividing $I$.  This probably implies that the measure is absolutely continuous.

\subsubsection{Transverse topology}   As usual for tilings of the line, the transversal $\transversal$ is the set of all tilings that have a tile endpoint at the origin.  Although it appears that the transversal might be connected, it is in fact a Cantor set.  

To show that any given tiling $\T \in \transversal$ is not isolated, consider any $\epsilon > 0$ and an $N$-supertile in $\T$ that contains $B_{2/\epsilon}(0)$.   (Even though this supertile may not be a natural supertile of $\T$, it must exist since $\T$ is admitted by the substitution rule).  There are lots of supertiles that are within $\epsilon/2$ of this supertile in the patch metric, since we can choose a supertile as close to the original in length as necessary to ensure they subdivide to comparable tiles.   Select one and let $\T' \in \transversal$ be a tiling with this supertile at the origin.   Then $d(\T, \T') < \epsilon$.

To show that the transversal is totally disconnected consider two tilings $\T, \T' \in \Tsp$ that are close, so that they very nearly have the same patch of tiles in a large ball $B$ around the origin.  All the tiles in $\T$ are multiples $2^j/3^k$ of each other, as are those in $\T'$.   Thus in $B$ all the tiles in each are multiples by the {\em same} $2^j/3^k$ 
of their respective tiles at the origin.   Suppose the tile for $\T$ at the origin is slightly larger than that in $\T'$.  All the tiles in $B$ in $\T$ are then larger than those of $\T'$.   Outside of $B$ there will be a tile in $\T'$ that is of size $3 - \epsilon$ such that in the corresponding region in $\T$ the $1$-supertile is of size slightly larger than $3$ and thus is broken into two tiles of sizes slightly greater than $1$ and $2$.   Using this difference we can partition the transversal into two clopen sets:  one containing all tilings with a tile of size less than or equal to 3/2 at the left side of this location and one containing all tilings with a tile of size greater than 3/2 there.

\subsubsection{Complexity}  Let $N(\epsilon, L)$ be the complexity function that counts the minimum number of patches of $d_L$-radius $\epsilon$ it takes to cover $\Tsp$.  Recall that a $d_L$-ball of radius $\epsilon$ is the set of all tilings that have a patch that agrees with a fixed tiling on $[-1/\epsilon, L + 1/\epsilon]$ up to $\epsilon$.   We concern ourselves with the situation where $L$ is much greater than $1/\epsilon$. 

In order to specify the fixed tilings we use as the `centers' of the balls, we need to specify up to $\epsilon$ all patches of size $L + 2/\epsilon$.   To do so, we need to find the smallest supertile that contains an interval of that size.   The length of that supertile is on the order of $L$ and we need to specify it up to $\epsilon$, so there are on the order of $L/\epsilon$ choices for that.  We also need to specify precisely where within the supertile we are up to $\epsilon$, and that gives us on the order of $L/\epsilon$ choices also.
This means that the complexity goes as $L^2/\epsilon^2$.   This is polynomial complexity, and the $\epsilon$-entropy is zero.

\subsection{Example \ref{DPVexample}: Direct product variation}
\label{DPV.analysis}
We begin our analysis of the DPV tiling space $\Tsp$ with an alternative description of how to obtain it by projection from a structure in $\R^3$. Then we show that it is minimal and show how to compute the invariant measure.   The topological spectrum and complexity are computed and the transversal is discussed for varying values of tile widths.   We conclude the section with a brief description of how the discussion would extend to other DPVs.

\subsubsection{Projection method}  A standard trick with one-dimensional substitution sequences is to construct the so-called `broken line' or `staircase' of a substitutive sequence.   The staircase lives in $\R^{|\aaa|}$ and is constructed iteratively by starting at the origin,  taking a step in the $\vece_{x(0)}$ direction, then one in the $\vece_{x(1)}$ direction, and so on.  It is well known that such a staircase will approximate the Perron eigenline of the substitution matrix, and projection of the staircase onto the Perron eigenline along the weak eigendirections yields the natural tile lengths for the corresponding self-similar tilings.  Projection of the staircase onto other lines or in other directions produces different tilings that may or may not be conjugate to each other, depending mostly (it seems) on the expansion factor of the system.  When the expansion matrix is Pisot all the points on the staircase lie within a bounded distance of the eigenline.   When it is not the staircase can wander arbitrarily far away from the eigenline, but it  comes close to the eigenline repeatedly.   %This is probably discussed further in the [French guys, Bernd and Marcy] chapter of this volume.

In a planar tiling context there are similar constructions when the expansion constant is Pisot \cite{ABS} or if other special conditions hold \cite{Gensubref}.   In those cases one constructs a `discrete plane' made up of two-dimensional facets in some $\R^n$.  The conditions in our examples do not satisfy these conditions, but nevertheless it is possible to construct a discrete surface that projects onto our DPV tilings.   This surface has holes that cannot be seen when we project in a direction that forms a tiling.  We explain the method for our basic example here and a description of how to generalize it appears at the end of this section.

The stepped surface will appear in $\R^3$ and can be constructed by substitution.   The $A = a \times c$ tile corresponds to the unit square spanned by the origin, $(1,0,0)$ and $(0,0,1)$.  The $B = b \times c$ tile corresponds to the unit square spanned by the origin, $(0,1,0)$, and $(0,0,1)$.  The substitution rule is pictured in figure \ref{stepped.subs}.

\begin{figure}[ht]
\parbox{1in}{\includegraphics[width=.75in]{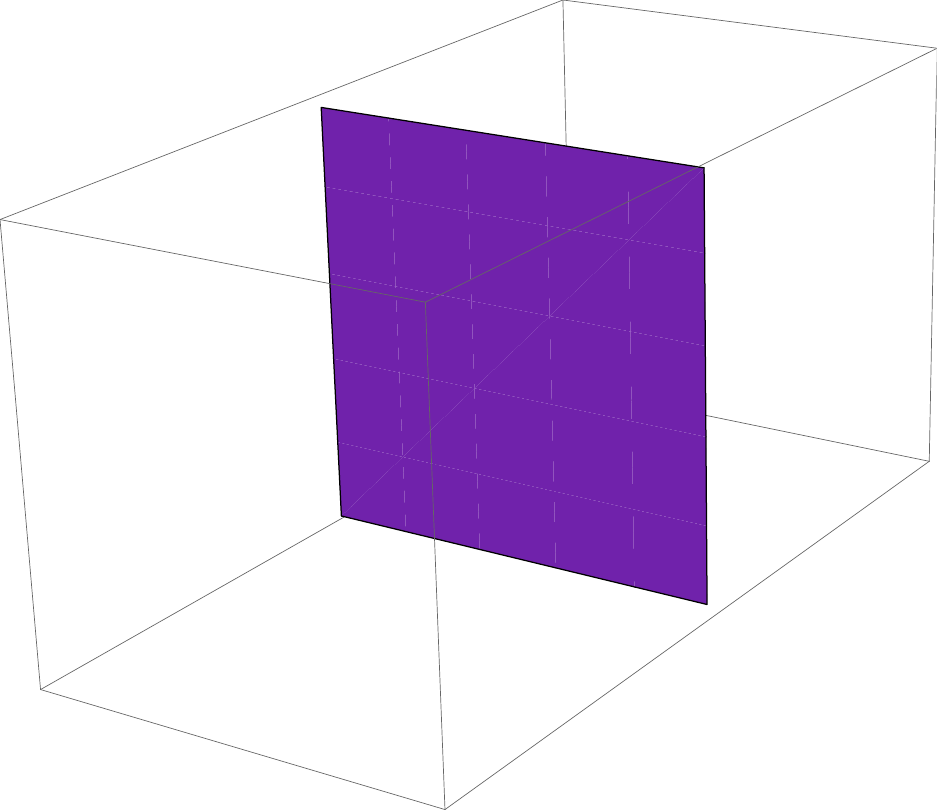}}
\raisebox{.2in}{\text{\Large $\to$}}
\raisebox{.2in}{\parbox{2in}{\includegraphics[width=1in]{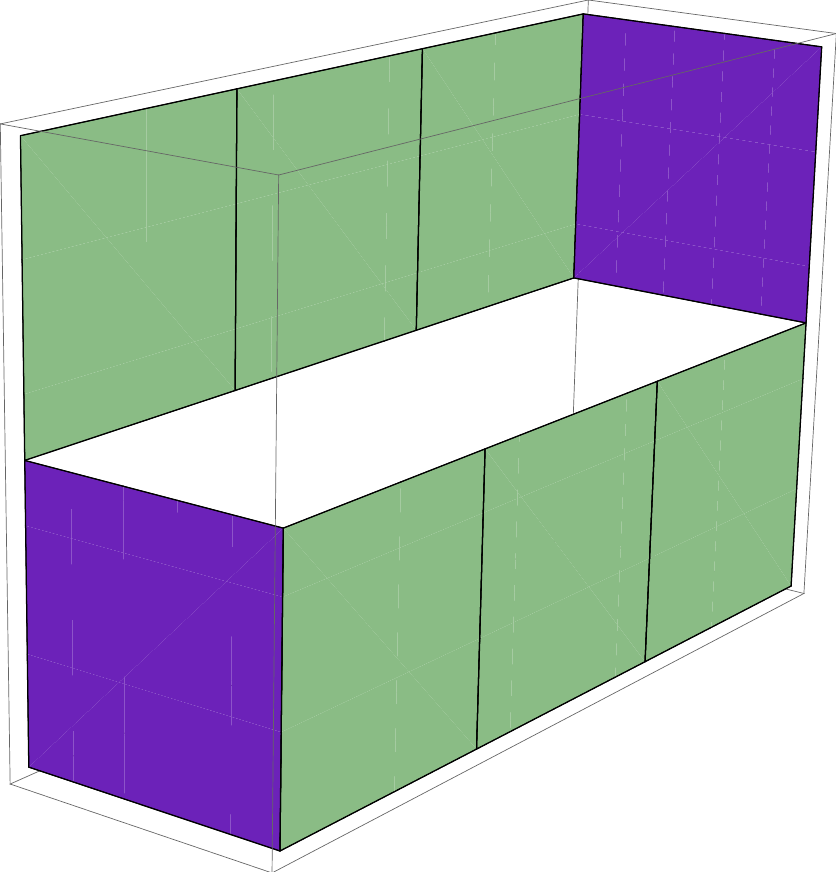}}}
%\hspace{1in}
\parbox{1in}{\includegraphics[width=.75in]{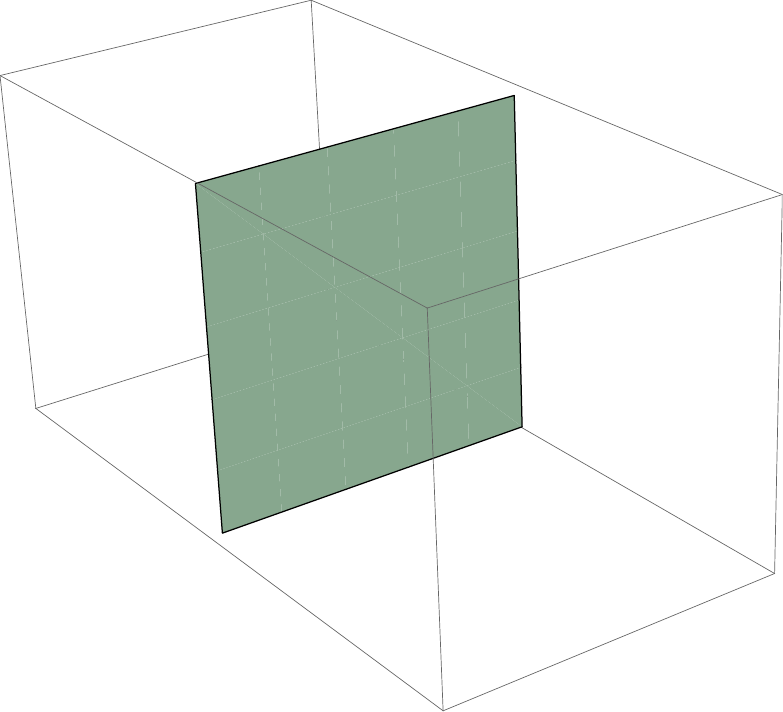}}
\raisebox{.2in}{\text{\Large $\to$}}
\raisebox{.2in}{\parbox{1.2in}{\includegraphics[width=.9in]{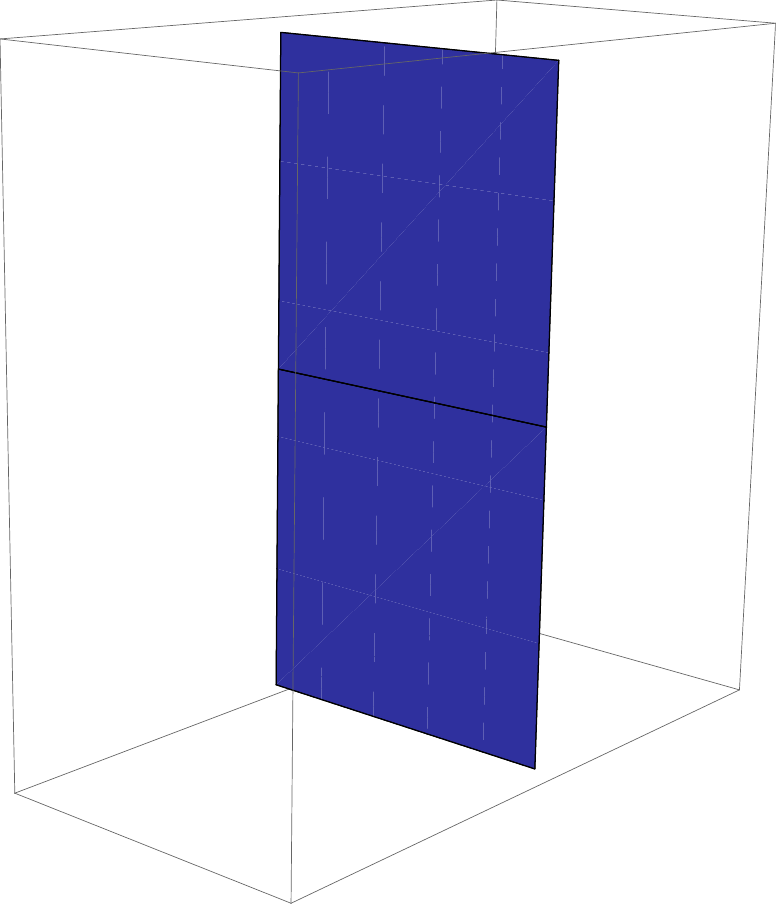}}}
\caption{The facet substitution.}
\label{stepped.subs}
\end{figure}

It can be proved that the substitution can be iterated indefinitely as long as one follows concatenation rules; alternatively the $n$-supertiles can be seen as fusions of the $(n-1)$-supertiles according to the same basic combinatorics as the original fusion.   In figure \ref{stepped.its} we show a $3$-supertile at an angle different than one of the projection angles we will use to see the planar tiling.    From this angle the holes in the surface are visible, and since the facet of the hypercube corresponding to an $a \times b$ is not an element of our tile set the holes are parallel to the $xy$-plane.  To obtain our planar DPV with parameters $a, b, c$, we simply take the stepped surface and project it using the matrix $\left(\begin{matrix}  a & b & 0\\ 0 & 0 & c\end{matrix}\right)$.

\begin{figure}[ht]
\includegraphics[width=4in]{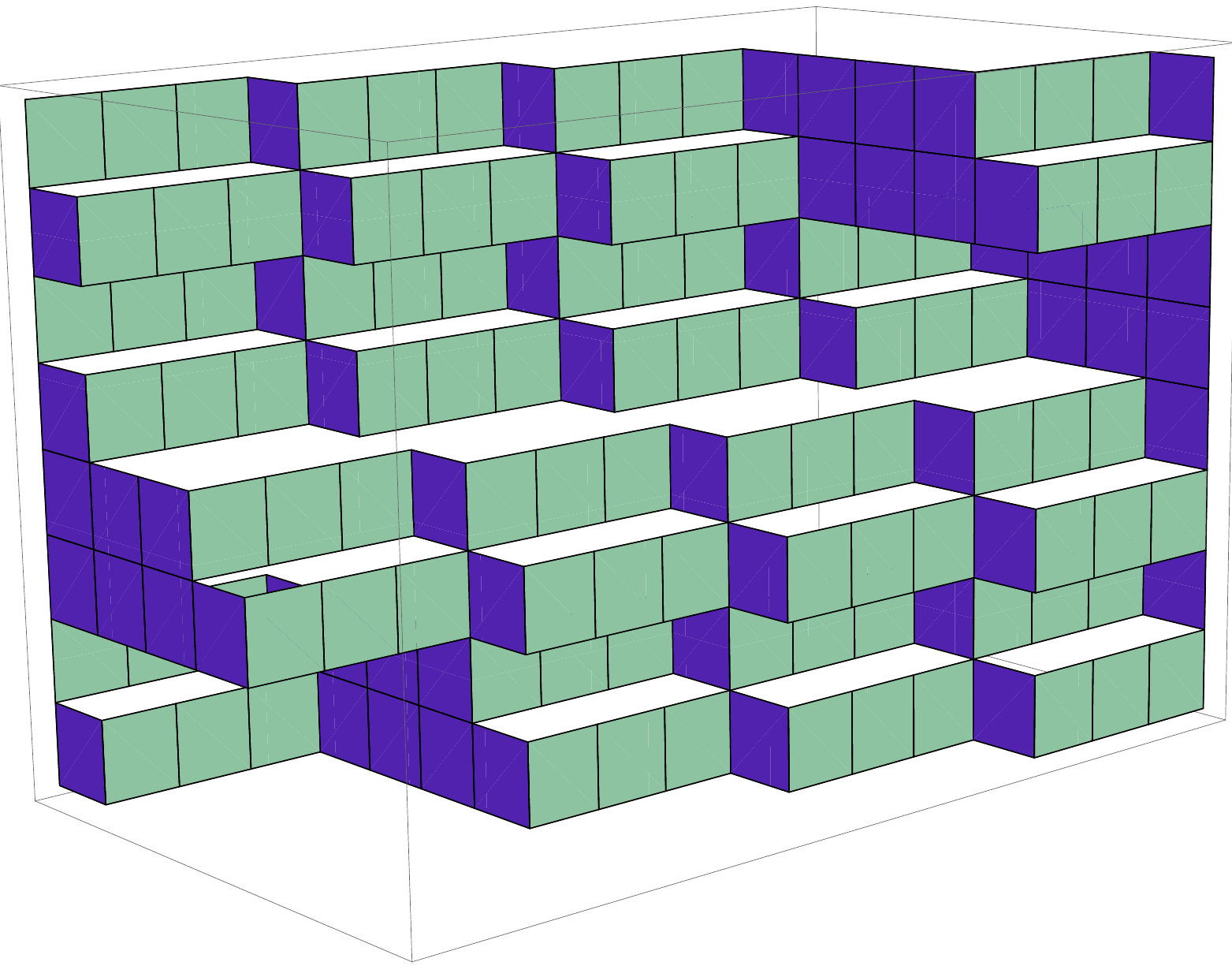}
\caption{The $3$-supertile of type $a \times c$.}
\label{stepped.its}
\end{figure}

Now we begin to see why the fault lines occur.  The holes correspond to pieces of fault line and we see a large hole halfway up the supertile that will continue to grow as we iterate the substitution.  In the limit the top and bottom halves become totally severed and this is reflected in $\Tsp$ by tilings with arbitrary shifts along the fault line.

\subsubsection{Minimality and invariant measure}  Our example has transition matrix $M = \left(\begin{matrix} 2 & 2 \\ 6 & 0\end{matrix}\right)$ as a standard substitution rule and so $\M_{n,N} = M^{N-n}$.    Since $M$ is a primitive matrix we immediately see that the fusion rule must be primitive and so its dynamical system is minimal.

Using results from \cite{Me.Lorenzo.ILCfusion} we know that for any width parameters $a$ and $b$ the dynamical system will be uniquely ergodic.  The frequency measure on patches is atomic and gives all patches that live in some $n$-supertile a positive frequency.   All patches that are admitted only in the limit and never appear in any $n$-supertile have measure 0.   

The precise nature of the ergodic measure depends on $a$ and $b$ for several reasons.   First off, if $a$ and $b$ are rationally related then the tiling space has finite local complexity.   The least common denominator $q$ determines the horizontal offsets, and all all multiples of $1/q$ occur.   This means that the patches that we are measuring the frequency of are different as we change the parameters.  If $a$ and $b$ are rationally independent then we have infinite local complexity occurring along horizontal fault lines (i.e., there is a fracture in the direction of $(1,0)$).  
Patches with frequency 0 appear.

  The Perron eigenvalue is $\lambda = 1 + \sqrt{13}$\footnote{which is the product of the length expansions of the one-dimensional substitutions.} and it has left eigenvector $\vecv = (\lambda, 6)$.    This vector, when normalized by volumes of $P_n(A)$ and $P_n(B)$, gives the frequencies $\rho_n(A)$ and $\rho_n(B)$ of those supertiles.     
The volume of $P_n(A)$ or $P_n(B)$ can be computed in two ways for a general $a$ and $b$.  Since $Vol(A) = ac$ and $Vol(B) = bc$, we see that
$(Vol(P_n(A)), Vol(P_n(B))) = (ac, bc)*M^n$.   On the other hand the volume of $P_n(A)$ or $P_n(B)$ is $2^n c$ times the length of the $n$th one-dimensional substitution of the $a$ or $b$ tile.   
To find $\rho_n$, we see that it is $(\lambda, 6)/K$, where $K = \lambda P_n(A) +  6P_n(B)$.
In the special case where $a$ and $b$ are natural tile lengths for the horizontal substitution, the tiling is self-affine and $K$ is a constant multiple of $\lambda^{-n}$. 

\subsubsection{Topological Spectrum}
No matter what our choice of $a$ and $b$ are there will be eigenvalues of the form $(0, c\Z[1/2])$.
This is because there are functions that detect only where a tiling is in the vertical, solenoid hierarchy.  Whether or not there is any horizontal spectrum depends on whether $a$ and $b$ are rationally related.

If $a$ and $b$ are irrationally related, the fault lines provide us with fractures of the form $(k,0)$ for arbitrarily small $k$.   By Corollary \ref{Spectrum.corollary} we see then that any eigenvalue must be of the form $(0,y)$, so the spectrum we retain is that of the solenoid.

If $a = pb/q$, then we obtain additional discrete spectrum of the form $(kp/a, 0)$, for $k  \in \Z$.  The eigenfunctions keep track of the horizontal offsets of, say, the tile containing the origin.   $f(\T)$ would equal the displacement of $\T$ horizonally from the nearest $1/q$ piece of the tiling.

\subsubsection{Transverse topology}  The topology of the transversal depends on whether $a$ and $b$ are rationally related.   If they are, then the tiling space has FLC and the transversal is a Cantor set.   However, when $a$ and $b$ are not rationally related the topology of the transversal is not well understood.   We show that the transversal is not a Cantor set and that in particular it is not totally disconnected.  For concreteness, assume the control points of the tiles are their centers of mass.  

Any tiling in the transversal that has a fault line is in a connected component of the transversal that contains all possible shifts of the half-plane that does not contain the origin.  Consider a point $\T \in \transversal$  that does not have a fault line.  It is arbitrarily close to tilings that do have fault lines: any ball around the origin is contained in some large supertile of $\T$, and that supertile can be contained in a tiling with a fault line.   Thus $\T$ exactly agrees with a continuum of tilings in $\transversal$ on this supertile.

In this example there are infinitely many connected components that are distinguished from one another by where the tiling is in the vertical, solenoid hierarchy.   However it is not clear whether a location in the hierarchy contains multiple connected components that are distinguished in some way by the horizontal substitution.  Understanding the transversal for tilings with both vertical and horizontal fault lines is even more mysterious.

\subsubsection{Complexity} The complexity of $\Tsp$ will depend on the parameters $a$ and $b$.   In the case where $a$ and $b$ are rationally related and $\Tsp$ has FLC we get a tiling space that has higher complexity than a planar self-similar tiling because of the unbounded nature of the holes.   For simplicity suppose $a$ and $b$ are integer multiples of $1/q$, for $q \in \N$.  Suppose that $L$ is greater than $1/\epsilon$ and both are greater than $q$. We need to count how many patches of size $L$ there are, up to $\epsilon$.  There is some minimum size of supertile for which every patch of size $L$ is contained entirely within a supertile or across the boundary of 2, 3, or 4 supertiles.   There are only a finite number of choices for such patterns since we only have two supertile types.  Inside of any such pattern we a bounded multiple of $L^2/\epsilon^2$ places to put the corner of a patch; moreover for any patch that has two or more supertiles there are about $qL$ ways they can fit together.   This implies that
the complexity as $L \to \infty$ is bounded above and below by bounded multiples of $ q L^3/\epsilon^2$.

Now suppose instead that $a$ and $b$ are irrationally related.   Most of the argument is the same except that now there are $L/\epsilon $ ways to fit two or more supertiles of size $L$ together.   This leaves us with a complexity bounded above and below by bounded multiples of $L^3/\epsilon^3$.   In either case we have polynomial complexity $C(\epsilon) L^3$, with the exponent $3$ being notable because it is greater than the ambient dimension of the tiling spaces.  Although the system still has entropy $0$, it is more complex than planar FLC self-similar tilings.

\subsubsection{Other DPVs}  Any DPV can be seen as the projection of a higher-dimensional stepped surface, but we will keep our discussion in the plane for simplicity.   If the alphabets are $\aaa$ and $\bbb$, then there are $|\aaa||\bbb|$ tile types, all of the form $a_i \times b_j$.  The unit hypercube in $\R^{|\aaa| + |\bbb|}= \R^\numa \times \R^\numb$ has $\left((|\aaa| + |\bbb|) \text{ choose }2\right)$ two-dimensional facets that contain the origin.   We need $\numa \numb$ of them for our stepped surface, which we construct as follows.  Let the standard basis vectors in $\R^\numa$ correspond to the letters in $\aaa$ and let the standard basis vectors of $\R^\numb$ correspond to those in $\bbb$, so that the tile $a_i \times b_j$ is represented by the corresponding facet that contains the origin, one standard direction in $\R^\numa$, and one in $\R^\numb$.   The remaining hypercube faces that contain the origin are those that take both other vectors in either $\R^\numa$ or $\R^\numb$.   These don't correspond to tiles in the DPV and so do not appear in the facet substitution (or rather, they appear as holes).   To give the tiling parameters $a_1, a_2, ... a_{\numa}, b_1, ... b_\numb$ we simply project it to the plane via the matrix $\left(\begin{matrix}  a_1 & \cdots & a_\numa &  0 & \cdots  &0 \\ 0 & \cdots & 0 &  b_1 & \cdots & b_\numb \end{matrix}\right)$

It is easy to construct DPVs that have fractures in both the horizontal and vertical directions by varying the direct product of two substitutions that have non-Pisot length expansions.  (The primary example in \cite{Me.Robbie} is of this form.)    By choosing irrationally related side lengths and irrationally related height lengths we obtain a topologically weakly mixing system.

\subsection{Example \ref{Solenoid.extensions}: Solenoid extensions}
\label{Solenoid.analysis}
Here we consider the tiling space $\Tsp_c$ obtained by the compact label set $\lcal = \N_0 \cup \lcal_c$.    For convenience of notation we will make a partial order on the label set: $k < m $ whenever that is true for $k$ and $m$ as integers, and also if $k$ is an integer and $m \in \lcal_c$.   Every integer label is `less than' every compactification label. 
We will see that on the measurable level the $\Tsp_c$s are all the same, but on the topological level they are quite different.

\subsubsection{Minimality} 
By looking at $P_4(k)$ for some $k \ge 4$ we can think about the transition map $\M_{n,N}: \ppp_n \to \ppp_N$.   
$$P_4(k) = P_3(k)P_3(3) = k0102010 \,30102010 $$
We have one 3 and one $k$; a pair of $2$s, four ones, and eight zeroes.    If instead we break it into 1-supertiles, we have  four of type 1,  two of type 2,  and one each of types $3$ and $k$.    We begin to suspect that the number of $n$-supertiles of type $m$, for $m \ge n$, in an $N$-supertile is equal to the number of tiles of type $m$ and has nothing to do with $N$ as long as $N$ is larger than $m$.    If $m$ is larger than $N$, then it is $0$ unless we are in an $N$-supertile that is also of type $m$.    In the former case we see that the number of $n$-supertiles of type $m$ in $P_N(k)$ is $2^{N-n}/2^{m-n+1}$ and we have

$$\M_{n,N}(m,k) =  \left \{\begin{array}{ccccc} 2^{N - (m+1)} & \text{ if } n \le m < k \text{ and } m < N  \\1 & \text{ if } N \le m = k  \\0 & \text{ otherwise }\end{array} \right.$$
 
For a fixed $n$ and $\epsilon > 0$, we can always choose an $N$ such that every $N$-supertile contains a tile arbitrarily close to $P_n(m)$ for any $m \ge n$.   The details of this will depend on the topology of $\lcal$, in particular on the speed with which sequences converge to elements of $\lcal_c$.   This proves primitivity and hence minimality.
 
\subsubsection{Measurable conjugacy of $\Tsp_c$ to the dyadic solenoid}
Every solenoid extension is measurably conjugate to the dyadic solenoid, defined as the inverse limit $\Sol= \lim_{\longleftarrow}(S^1, \times 2)$.   Elements of $\Sol$ take the form $(x_0, x_1, x_2, ...)$ such that for each $n$, $2x_n \equiv x_{n-1} \mod 1$.   We define $f(\T) = (x_0, x_1, x_2, ...)$ as follows.  The position of the origin in the tile containing it determines $x_0$, choosing $x_0 = 0$ if the origin is on the boundary between two tiles.  The origin is either in the left or right half of the 1-supertile that contains it.  If it is in the left half we let $x_1 = x_0/2$ and if it is in the right half or exactly in the middle we let $x_1 = (x_0 + 1)/2$.  Once this is determined we look at whether the origin is in the left half of the 2-supertile containing it or if it is in the right (or exactly in the middle).    In the former case we let $x_2 =  x_1/2$ and in the latter we let $x_2 = (x_1 + 1)/2$.  Since the origin lies somewhere in an $n$-supertile for every $n \in \N_0$ we can inductively define $f$ for all of $\Tsp_c$.

Every tiling $\T \in \Tsp_c$ that doesn't contain any element of $\lcal_c$ is sent to a unique point in $\Sol$.   Since such tilings form a set of full measure we have measurable conjugacy.   The rest of the tilings in $\Tsp_c$ do not map in a one-to-one fashion since a tiling that is made of two infinite-order supertiles maps to an element of the solenoid that knows the position of the origin in its supertile but not what element of $\lcal_c$ the tiling contains. This also explains why the one-point compactification is topologically conjugate to the solenoid.

%since two tilings that are always in the right half of their supertiles and have the same $x_0$ can fail to agree on the right half of $\R$.  Tilings that contain such a  but these form a set of measure 0 since the frequency of any limit label is $0$. 

Since $\Tsp_c$ is measurably conjugate to the dyadic solenoid it must be uniquely ergodic and share its purely discrete measurable spectrum $\Z[1/2]$.   The frequency measures are supported on the subset of the transversal on which $f$ is one-to-one, i.e. no infinite-order supertiles.   However, the nature of the limits does affect the nature of the measurable isomorphism.    A ball of radius $\epsilon$ around a tiling with a given element of $\lcal_c$ at the origin maps onto a set of equal measure in the solenoid, but this set will vary as the elements of $\lcal_c$ do, or when the sequences that converge to them do.

\subsubsection{The topological impact of $\lcal_c$}
To begin we consider the one-point compactification that has elements of $\N_0$ converging monotonically to a single limit point $\ell$.  In this case we can see that $\Tsp_c$ is topologically conjugate to the dyadic solenoid using the argument from the previous section on measurable conjugacy.   In that argument the map is one-to-one for each tiling in $\Tsp_c$ that does not contain an element of $\lcal_c$ but multiple-to-one on those that do.  However, the multiple copies are in correspondence to the number of elements of $\lcal_c$.  Since in this case there is only one element of $\lcal_c$ the map continues to be one-to-one on the exceptional tilings, so it is one-to-one everywhere and thus provides a topological conjugacy to the dyadic solenoid. 

This gives us the unusual situation where the tiling dynamical system is not expansive.  To see this, choose any $\delta > 0$ and find an $N$ such that if $N', N'' > N$, then the distance between prototiles labelled $N'$ and $N''$ is less than $\delta$.   Any two tilings $T$ and $T'$ for which the origin sits in the same location in its $N$-supertile will then satisfy the property that $d(T-x, T'-x) < \delta$ for all $x \in \R$ regardless of the type of the $N$-supertile.  The way to think about this is that each $N$-supertile differs from any other only on the first tile, and those first tiles are labelled greater than $N$ and thus differ by less than $\delta$.

The situation for two-point compactifications becomes more subtle. The version for which all even numbers converge monotonically to $\ell$ and all odd numbers converge monotonically to $\ell'$ is shown to be topologically conjugate to the FLC tiling space generated by the `period-doubling' substitution $X \to YX, Y \to XX$, where $X$ and $Y$ are unit length tiles.
If instead $\N_0$ is partitioned into two sets $S$ and $S'$, one converging to $\ell$ and the other to $\ell'$,  the measures of clopen subsets of the transversal are determined by  $\alpha = \sum_{n \in S} 2^{-n}$, and thus so is the gap-labelling group.  Two tiling spaces, one with $\alpha$ and the other with $\alpha'$ can be compared, and whether or not they are topologically conjugate or even homeomorphic depends on how these constants are related.  See  \cite{Me.Lorenzo.ILCfusion}  for details.

% In order to compare two two-point compactifications, we need look at the values they get for $\alpha$. If they differ by a dyadic rational number then the tiling spaces are seen to be topologically conjugate because the convergent sequences agree on their tails.  If the difference is by a non-dyadic rational then the spaces are not topologically conjugate but are homeomorphic.   If the difference is not rational the tiling spaces are not even homeomorphic.   Nonetheless the spaces all have the same cohomology.

By compactifying with limit sets that have an interesting topology, we can obtain tiling spaces with nontrivial cohomologies in $H^n$.   These examples can be distinguished by their cohomology.

\subsubsection{Transverse topology} The topology of the transversal depends on the topology of $\lcal_c$. 
 Suppose that  $\T,\T'\in \transversal$ are two tilings that have an element of $\lcal_c$ at the origin.   Note that $\T$ and $\T'$ are made of two infinite-order supertiles and thus can only differ at the origin.   So if it were possible to make two open sets in $\transversal$ that disconnected the transversal and contain one each of $\T$ and $\T'$, those sets would correspond to two open sets in $\lcal_c$ that disconnected $\lcal_c$.    Conversely, if $\lcal_c$ can be separated by two open sets, then if $\T$ and $\T'$ contain an element of these sets at the origin, we can partition the rest of the elements of $\lcal$ so that they are close to one or the other of the subsets of $\lcal_c$ to create a pair of open sets that disconnect $\transversal$ and contain $\T$ and $\T'$ respectively.  
 Thus $\transversal$ is disconnected if and only if $\lcal_c$ is.

In particular, the one- and two-point compactifications of the solenoid have totally disconnected transversals.   The one-point compactification is not strongly expansive because it is not expansive.   To see this, consider a fixed $\delta > 0$ and let $d(\T_1,\T_2) \le \delta$, where $\T_1$ and $\T_2$ have an $N$-supertile beginning at the origin, where $N$ is sufficiently large that any two elements of $\lcal_c$ that are greater than $N$ are within $\delta$ in the tile metric.  Then $\T$ and $\T'$ differ at most on the beginning of each $N$-supertile, but these are all within $\delta$ of one another in the patch metric.   Thus $d(\T - k, \T' - k) < \delta$ for all $k \in \Z$ and the dynamics on the transversal are not expansive (and indeed the overall dynamical system is not expansive).  Thus Theorem \ref{One-D-FLCresult} confirms that the one-point compactification is not homeomorphic to any FLC tiling system. 

By a similar argument we can show that the two-point compactification is homeomorphic to a FLC tiling space by Theorem \ref{One-D-FLCresult}.   This is because even if two tilings very nearly agree on the central $N$-supertile, each of them are bound to have $N$-supertiles whose first tiles are close to distinct elements of $\lcal_c$ and thus not to each other.

If $\lcal_c$ is not totally disconnected, then its solenoid extension is not homeomorphic to a one-dimensional tiling space with finite local complexity.

\subsubsection{Complexity}  \label{Solenoid.Complexity}
This is the topological invariant that lets us see that the one-point compactification, while failing to have finite local complexity, does so in a way that is much less complex than a non-periodic FLC tiling space.  In order to be specific about complexity let us use $N_1(\epsilon, L)$, the minimum number of $d_L$- balls of radius $\epsilon$ balls it takes to cover $\Tsp$.    Take $N \in \N_0$ such that any element $n \ge N$ is within $\epsilon$ of the limit point $\infty$ in $\lcal$.  Every $n$-supertile with $n \ge N$ is within $\epsilon $ of any other in the patch metric.  This means that that the system is basically periodic with period $2^N$, up to $\epsilon$, and $N_1(\epsilon, L) $ is bounded multiple of $2^N$ once $L$ is sufficiently large relative to $\epsilon$.  This means that the system has bounded complexity in the sense of Section \ref{entropy.complexity}.

No matter how complicated $\lcal_c$ is, the complexity cannot become too high and there will always be zero entropy.  This is because, for any $\epsilon > 0$, an $\epsilon-$cover of $\lcal$ into $N$ elements effectively reduces $\Tsp_c$ to a solenoid extension on $N$ elements.

\section{Important questions}

There are two main types of questions that seem to be important in the study of tiling spaces with infinite local complexity.   One is to ask, what properties of these spaces
are invariant under homeomorphism or some other type of conjugacy? Another is to ask, how do combinatorial and geometric factors influence the dynamical, measure-theoretic, topological, or complexity properties?   

It was already known from \cite{Radin-Sadun} that infinite local complexity is not a topological invariant.   The solenoid extensions considered here show that infinite local complexity is not preserved by measurable or topological conjugacy either.  So what are the important classes of infinite local complexity?
Even though FLC isn't an invariant property, there should be some properties that guarantee that a tiling space is topologically conjugate, or measurably conjugate, or just homeomorphic to one that is FLC.  In one dimension, Theorem \ref{One-D-FLCresult} gives a topologically-invariant characterization of this class.  In higher dimensions geometry becomes an obstacle and a direct generalization is unlikely to be sufficient: the methods used in the one-dimensional proof don't produce valid tilings in higher dimensions.   So the topological classification of FLC in 2 or higher dimensions remains unresolved, along with the dynamical classification in any dimension.

The property of being constructed from a finite set of prototiles is also not invariant.    A simple example is the tiling space made out of unit square tiles that lie in rows that are offset by random amounts.  We can introduce an infinite number of tile types symbolically, by constructing an infinite label set given by all the possible coronas, which produces
a topologically conjugate tiling space.    Alternatively, we can introduce an infinite number of tile types geometrically by taking the space of dual tilings.   If we use the centers of the tiles as the vertices of the dual prototiles, the dual prototile set consists of one square and infinitely many triangles  of equal areas but different angles, and is not even compact.  Since many important tiling models assume a finite prototile set, knowing properties that guarantee that a space is equivalent to one with this property seems essential.
  
The combinatorics in the tilings in the last paragraph are all basically the same and depend only on whether an offset between rows is zero or not.   Thus the combinatorics of that particular tiling space do not reveal information about local complexity.  By way of contrast, the combinatorics of the DPV of example \ref{DPVexample} will determine whether the tiling space has finite or infinite local complexity.  The interplay between combinatorics and complexity is especially evident in the DPV case.   If one does not vary the direct product substitutions, no geometric factors can influence the local complexity of the tiling space:  it is guaranteed to have FLC.   But if one does vary the direct product to produce a DPV, then the local complexity depends on several factors:   the inflation constant, the specific combinatorics of the substitution, and the sizes of the tiles.

More generally, in the category of primitive fusion or substitution tilings the interplay between combinatorics, geometry, number theory, topology, and dynamics is an important area of investigation.   It is known that some form of Pisot condition on the expansion factor has a profound effect on dynamics and complexity (cf. \cite{Sol.self.similar,Me.Robbie,Clark-Sadun}).   Such Pisot conditions enforce a rigidity on the tiling space; the dynamics and complexity are severely restricted.   In the absence of a Pisot condition, combinatorics and geometry can influence the tiling space in numerous ways.   In order to understand certain tiling models it is important that we understand the nature of this influence.

%
%
%
%\subsection{ILC planar tilings with infinite prototile set}
%
%\begin{ex} blurry crystal:  put a random variable around each lattice point.
%
%\end{ex}
%
%\begin{ex}{\em Ir-rep-tiles}
%
%\end{ex}
%
%\begin{ex}{\em Generalized pinwheel.}  30-60-90 generalized pinwheel or some other generalized pinwheel with finitely many angles but infinitely many tile sizes.
%
%\end{ex}
%
%\begin{ex}{\em The dual tiling of an ILC DPV.}
%
%\end{ex}

\end{document}